\documentclass[11pt]{article}
\usepackage[margin=1in]{geometry}  % set the margins to 1in on all sides
\usepackage{graphicx}              % to include figures
\usepackage{amsmath}               % great math stuff
\usepackage{amsfonts}
\usepackage{amssymb}       % for blackboard bold, etc
\usepackage{amsthm}                % better theorem environments
\usepackage{amssymb} 
\usepackage{enumitem}
\usepackage{color, soul}	
\usepackage{hyperref}
\usepackage[titletoc,title]{appendix}
\usepackage{algorithm}
\usepackage{algpseudocode}
\hypersetup{
	colorlinks=true,
	urlcolor=blue,
	linkcolor=blue,
	citecolor=red
	}

\newtheorem{thm}{Theorem}[section]
\newtheorem{lem}[thm]{Lemma}
\newtheorem{prop}[thm]{Proposition}

\newtheorem{conj}[thm]{Conjecture}
\newtheorem{defn}[thm]{Definition}

\newtheorem{clm}[thm]{Claim}
\newtheorem{qst}[thm]{Question}

\newcommand{\cl}[1]{\mathcal{#1}} % for caligraphic symbols
\newcommand{\EE}{\mathbb{E}}	  % for expectation

\newcommand{\reg}{\textnormal{reg}}
\renewcommand{\deg}{\textnormal{deg}}

\makeatletter
\renewcommand*{\@fnsymbol}[1]{\ensuremath{\ifcase#1\or *\or 
		\mathsection\or \mathparagraph\or \|\or **\or \dagger\dagger
		\or \ddagger\ddagger \else\@ctrerr\fi}}
\makeatother

\begin{document}

\title{Client-Waiter games on complete and random graphs}

\author{Dean, Oren\thanks{School of Mathematical Sciences, Raymond and Beverly Sackler Faculty of Exact Sciences, Tel Aviv University,
		Tel Aviv, 6997801, Israel. Email: orendean@mail.tau.ac.il.} \and Krivelevich, Michael\thanks{School of Mathematical Sciences, Raymond and Beverly Sackler Faculty of Exact Sciences, Tel Aviv University, Tel
		Aviv, 6997801, Israel. Email: krivelev@post.tau.ac.il. Research supported in part by USA-Israel BSF Grant 2014361
		and by grant 912/12 from the Israel Science Foundation.}} 

\maketitle

\begin{abstract}
For a graph $ G $, a monotone increasing graph property $ \cl{P} $ and positive integer $ q $, we define the Client-Waiter game to be a two-player game which runs as follows. In each turn Waiter is offering Client a subset of at least one and at most $ q+1 $ unclaimed edges of $ G $ from which Client claims one, and the rest are claimed by Waiter. The game ends when all the edges have been claimed. If Client's graph has property $ \cl{P} $ by the end of the game, then he wins the game, otherwise Waiter is the winner. In this paper we study several Client-Waiter games on the edge set of the complete graph, and the $ H $-game on the edge set of the random graph. For the complete graph we consider games where Client tries to build a large star, a long path and a large connected component. We obtain lower and upper bounds on the critical bias for these games and compare them with the corresponding Waiter-Client games and with the probabilistic intuition. For the $ H $-game on the random graph we show that the known results for the corresponding Maker-Breaker game are essentially the same for the Client-Waiter game, and we extend those results for the biased games and for trees.
\end{abstract}

\section{Introduction}
Positional games are games of complete information with no random moves. The inception of the study of positional games goes back to the seminal papers of Hales and Jewett \cite{HJ}, of Lehman \cite{Lehman}, and of Erd\H{o}s and Selfridge \cite{ES}. It has developed since to be a recognized area of combinatorics, and the last decade has seen a burst of papers and an increasing interest in this field (see, for example, the monograph of Beck \cite{BECK08} and the recent monograph \cite{HKSS14}).

Several variants of positional games have been considered in the literature. In the classical \textbf{Maker-Breaker} game the two players ('Maker' and 'Breaker') alternately claim elements from  a set $ X $ (\emph{the board}). Maker wins the game if he fully claims some set from a predefined family $ \cl{F}\subseteq \cl{P}(X) $, which we call the \emph{winning sets}, and Breaker wins otherwise. 

Recently, some attention was turned to another type of game, namely the \textbf{Waiter-Client} game. In this variant, initially defined and studied by Beck (e.g. \cite{BECK02}) under the name ``Picker-Chooser'', in each turn Waiter is picking a subset of $ q+1 $ free elements, where $ q $ is a fixed positive integer called \emph{the bias}. Client then needs to choose one element from this subset which he claims, while Waiter claims the remaining $ q $ elements. If there are less than $ q+1 $ remaining elements then in the last turn Waiter will claim all of them. Waiter's goal is to force Client to fully claim a winning set while Client tries to avoid it. It can be easily shown that this game is \emph{bias monotone}, i.e. for two positive integers $ q_1<q_2 $, a winning strategy for Waiter when playing with bias $ q_2 $ implies a winning strategy when playing with bias $ q_1 $. We can therefore define the \emph{critical bias}, $ q_c $, to be the unique integer for which Waiter has a winning strategy if and only if $ q\leq q_c $.

The \textbf{Client-Waiter} game runs much the same as the Waiter-Client with two differences. The first is that Client now tries to claim a winning set and Waiter tries to prevent it. The second is that we introduce a special monotonicity rule which states that Waiter may offer any number of elements between 1 and $ q+1 $ in a turn. Client still claims one element in each turn (including the last one). The motivation for this rule is the fact that without it the game is not bias monotone. (Consider for example a game with $ n $ pairwise disjoint winning sets of size 2. In this game Waiter wins whenever $ q+1 $ is even and loses otherwise.) The \emph{critical bias}, $ q_c $, in this game can thus be defined to be the unique integer for which Client has a winning strategy if and only if $ q\leq q_c $.

We denote by $ WC(X,\cl{F},q) $ and $ CW(X,\cl{F},q) $ the Waiter-Client and  Client-Waiter games, with board $ X $, winning sets $ \cl{F} $ and bias $ q $. Usually when considering a Waiter-Client or a Client-Waiter game we are interested in finding, or at least bounding, the critical bias. Another question we might ask is what is the probability threshold for the property  ``Waiter wins $ WC(X_p,\cl{F}_p,q) $" or ``Client wins $ CW(X_p,\cl{F}_p,q) $", where $ X_p $ is a random subset of $ X $ generated by removing each element from $ X $ randomly and independently with probability $ 1-p $ and $ \cl{F}_p $ is the subfamily of $ \cl{F} $ which includes only the sets of $ \cl{F} $ which are subsets of $ X_p $. 

Besides the interest in those games for their own right it was observed that in many cases they exhibit a strong probabilistic intuition. That is, the outcome of many Waiter-Client and Client-Waiter games is roughly the same as what we would expect it to be when both players just play randomly (although a random strategy for a single player is usually far from being optimal). See for example \cite{BECK02,BECK08, BiasedPC,CP_Games,Man_waiter,PC_H,WC_CW_ham,WC_CW_Plan}.

In this paper we look into Client-Waiter games played on the edge set of a graph $ G $, where $ G $ is either $ K_n $ --- the complete graph on $ n $ vertices, or $ G_{n,p} $ --- the random graph generated by taking every edge of $ K_n $ to be in the graph randomly and independently with probability $ p $. The $ G_{n,p} $ model is the most commonly studied probability distribution on graphs (see \cite{B01}, \cite{JLR00}, and the most recent \cite{FK15}). \\
For an infinite series of events $ \{ A_n \}_{n\geq 1} $ we say that $ A_n $ happens \emph{with high probability (w.h.p.)} if $ \lim\limits_{n\to\infty}\Pr[A_n]=1 $. Let $ \cl{P} $ be a monotone increasing graph property. We say that $ G_{n,p} $ goes through a \emph{phase transition} around $ p^* $ with relation to $ \cl{P} $ if w.h.p. $ G_{n,p}\notin\cl{P} $ whenever $ p\ll p^* $, while w.h.p. $ G_{n,p}\in\cl{P} $ when $ p\gg p^* $. \\
For games on $ K_n $ we bound the critical bias for a few games where Client tries to achieve some monotone graph theoretic property. The interpretation of the probabilistic intuition in this case is that the critical bias should be $ q_c\approx 1/p^* $, where $ p^* $ is the value around which $ G_{n,p} $ goes through a phase transition with relation to this property.  \\
For a fixed graph $ H $ we define \emph{the $ H $-game}, denoted $ CW(G,H,q) $, to be the Client-Waiter game on the edge-set of $ G $ with bias $ q $, where Client's goal is to build in his graph a copy of $ H $. We investigate the $ H $-game on the edges of $ G_{n,p} $ and find the value of $ p^* $ for the property ``Client wins $ CW(G_{n,p},H,q) $".

Not many biased Client-Waiter games on graphs have been studied to date. Bednarska-Bzd\c{e}ga, Hefetz and \L{}uczak  (\cite{BiasedPC}) showed that in a Client-Waiter game on the edges of $ K_n $ with bias $ q=(1+o(1))n/\ln n $ Waiter can isolate a vertex in Client's graph, while if the bias is $ q=(1-o(1))n/\ln n $ then Client can guarantee his graph will be $ k $-vertex connected, or alternatively he can guarantee his graph is Hamiltonian. This fits very well with the probabilistic intuition as $ G_{n,p} $ goes through a phase transition with relation to the properties of being $ k $-vertex connected and Hamiltonian around $ p^*=\ln n/n $ (see, for example, Chapters 4.2 and 6.2 of \cite{FK15}). Recently Hefetz, Krivelevich and Tan (\cite{WC_CW_Plan}) analysed the non-planarity, $ K_t $-minor and non-$ k $-colorability Client-Waiter games. They showed that these games also exhibit some probabilistic intuition (though not as strong as in the former games).

We start by considering the \textbf{maximum-degree game}, i.e. the game in which Client tries to claim a star of maximum possible size. For integers $ n,q $ let $ \mathcal{S}(n,q) $ denote the size (no. of edges) of a largest star graph Client can build when playing a Client-Waiter game with bias $ q $ on $ E(K_n) $.
\begin{prop}\label{prp:star}
	For any positive integers $ n,k $, if $ q\geq \lceil n/k\rceil-2 $ then $ \mathcal{S}(n,q)\leq 2k $, while for $ k\geq 2 $, if $ q< \dfrac{n-1}{k-1}-1 $ then $ \mathcal{S}(n,q)\geq k $.
\end{prop}
Notice that the claim in the above proposition is not asymptotic, and $ k $ can be fixed or it can be a function of $ n $. 

In the \textbf{large component game} Client tries to build in his graph a connected component as large as possible, while in the \textbf{path game} he tries to build a path as long as possible. Let $\mathcal{C}(n,q),\mathcal{P}(n,q) $ denote the size of a largest component, and the length of a longest path (no. of edges), respectively, that Client can build in the Client-Waiter game on the edges of $ K_n $ with bias $ q $. 
\begin{thm}\label{thm:component}\text{}
	\begin{enumerate}[label=\normalfont\bfseries(\roman*)]
		\item For every $ n $ and for every $ k>0 $, if $ q\geq 6n^{2^k/(2^k-1)} $  then $ \mathcal{C}(n,q)<3^k $.
		\item For every $ \epsilon>0 $ and $ n $ large enough, if $ q\geq 1.6n $ then $ \cl{C}(n,q)< (\ln n)^{\log_23+\epsilon} $.
		\item  For every $ 0<\epsilon<1 $ and $ n $ large enough, if $ q\leq(1-\epsilon)\dfrac{n}{2} $ then $ \mathcal{C}(n,q)\geq e^{-5/2\epsilon+3/2} n$.
	\end{enumerate}
\end{thm}

\begin{thm}\label{thm:path}\text{}
	\begin{enumerate}[label=\normalfont\bfseries(\roman*)]
		\item For every $ n $ and for every $ k>0 $, if $ q\geq 3n^{2^k/(2^k-1)} $  then $ \mathcal{P}(n,q)<2k $.
		\item If $ q\geq n $ and $ n $ is large enough then $ \mathcal{P}(n,q)<3\ln\ln n $.
		\item For every $ \epsilon>0 $ small enough  there is $ n_0 $ such that for every $ n>n_0 $, if $ q\leq(1-\epsilon)\dfrac{n}{2} $ then $ \cl{P}(n,q)\geq e^{-12/\epsilon}n $.
		\item For every $ \epsilon>0 $ small enough there is $ n_0 $ such that for every $ n>n_0 $, if $ q\leq\epsilon n $ then $ \cl{P}(n,q)\geq (1-4\epsilon\ln(1/\epsilon))n $.
\end{enumerate}
\end{thm}
We did not find a matching lower bound for Theorems~\ref{thm:component}(i) and \ref{thm:path}(i). We conjecture the following.
\begin{conj}
	For every positive integer $ k $ and every constant $ C>0 $, Client wins $ CW(K_n,P_k, Cn) $, provided $ n $ is large enough.
\end{conj}

Bednarska-Bzd\c{e}ga proved in \cite{CP_weight} several criteria for Waiter's win in biased Client-Waiter games and as an application showed the 
following bounds on the critical bias in the Client-Waiter $ H $-game.
\begin{thm}[6.1 in \cite{CP_weight}]\label{thm:H-game on K_n}
	Let $ H $ be a graph with at least two edges. For every $ 0<\epsilon<1 $ and $ n $ large enough the following holds. If $ q\geq n^{1/m'(H)+\epsilon} $ then Waiter wins $ CW(K_n, H, q) $ while if $ q\leq n^{1/m''(H)-\epsilon} $ Client wins $ CW(K_n, H, q) $, where 
	\begin{align*}
		m'(H)&=\max_{H'\subseteq H:v_{H'}\geq 1}\dfrac{e_{H'}-1}{v_{H'}},\\
		m''(H)&=\max_{H'\subseteq H:v_{H'}\geq 3}\dfrac{e_{H'}+1}{v_{H'}-2}.
	\end{align*}
\end{thm}

Here we give an improvement of the lower bound using the hypergraph containers result of Saxton and Thomason (\cite{h_containers}). The idea for this proof was suggested to us by Bednarska-Bzd\c{e}ga.
\begin{prop}\label{prp:k-clique lower bound}
	For every graph $ H $ there is $ c>0 $ such that Client wins $ CW(K_n,H, q) $ whenever $ q\leq cn^{1/m_2(H)}/\ln n $.
\end{prop}

In our treatment of the Client-Waiter games on the random graph we are very much influenced by the proofs of Nenadov, Steger and Stojakovi\'c (\cite{MBH}) for the unbiased Maker-Breaker $ H $-game played on random graphs. We show that the results in \cite{MBH} are also true in the Client-Waiter game. Moreover, we extend these results for any fixed bias and show that the case of $ H=K_3 $, which was an exception in the unbiased game, is no longer such when $ q\geq 2$.

The \emph{density} of a graph $ G $ is defined to be $ d(G)=e_G/v_G $, and its \emph{2-density} to be $ d_2(G)=(e_G-1)/(v_G-2) $ or 0 if $ v_G\leq 2 $. The \emph{maximum density} of $ G $ is $ m(G)=\max\limits_{G'\subseteq H} d(G') $, and its \emph{maximum 2-density} is $ m_2(G)=\max\limits_{G'\subseteq H} d_2(H') $. We say that a graph $ G $ is \emph{balanced} (resp. \emph{2-balanced}) if $ d(G)=m(G) $ (resp. $ d_2(G)=m_2(G) $). If for any proper subgraph $ G'\subset G $ we have $ d(G')<m(G) $ (resp. $ d_2(G')<m_2(G) $) then $ G $ is \emph{strictly balanced} (\emph{strictly 2-balanced}).

\begin{thm}\label{thm:CW_random_H_biased}
	Let $ H $ be a graph which is not a forest. If either $ q\geq 2 $ or there exists $ H'\subseteq H $ such that $ d_2(H')=m_2(H) $, $ H' $ is strictly 2-balanced and it is not a triangle, then there exist constants $ c,C>0 $ which depend only on $ H $ and $ q $ such that in the Client-Waiter game $ CW(G_{n,p},H,q) $ 
	\[ \lim\limits_{n\to\infty} \Pr[\text{Client wins}]=\begin{cases}
	1, &p\geq Cn^{-1/m_2(H)}, \\
	0, &p\leq cn^{-1/m_2(H)}. \\
	\end{cases} \]
\end{thm}
	Note: the same is true for the corresponding Maker-Breaker game. 
	
For the case $ H=K_3 $ and $ q=1 $, which is missing in the above theorem, we have the following theorem, whose proof is much the same as Theorem 1.3 of \cite{MBClique} for the corresponding Maker-Breaker game and is therefore not included here.
\begin{thm}\label{thm:K_3 hitting time} For every $ p=p(n) $,
\[ \lim\limits_{n\to\infty}\Pr[\text{Client wins } CW(G_{n,p},K_3,1)]=\lim\limits_{n\to\infty}\Pr[G_{n,p}\text{ contains } K_5-e]. \] 
\end{thm}

Lastly, we prove that for trees Client can win even when $ p\ll n^{-1} $.
\begin{prop}\label{prp:CW_random_tree}
	For every positive integers $ k,q $  w.h.p. Client has a strategy to win $ CW(G_{n,p},T_{k,k},q) $, where $ T_{k,k} $ is the complete $ k $-ary tree of height $ k $, and $ p= n^{-1-(k(q+1))^{-2(k+1)}} $.\\ 
\end{prop}
	Note: the same is true for the corresponding Maker-Breaker game.

\subsection{Related results}
In a groundbreaking paper (\cite{ER60}), Erd\H{o}s and R\'{e}nyi showed that the random graph $ G_{n,p} $ goes through a phase transition around $ p=1/n $, from typically having only connected components of size at most logarithmic in $ n $ to having a linear sized (`giant') component. Later Ajtai, Koml\'os and Szemer\'edi (\cite{AKS81}) proved that in the supercritical regime $ p=(1+\epsilon)/n $, $ G_{n,p} $ will typically contain not only a giant component, but a linearly-long path (see also \cite{DFS} for a simple proof).
Theorems~\ref{thm:component}(ii)+(iii) and~\ref{thm:path}(ii)+(iii) show that a similar transition takes place in the Client-Waiter game when $ q $ is between $ n/2 $ and $ 1.6n $ for components and between $ n/2 $ and $ n $ for paths. Moreover, when $ q=n $ Waiter can even limit Client's longest path to be of order $ \ln\ln n $ which indicates some weakness of Client in comparison to the probabilistic intuition. A similar Client's weakness can be observed in Theorem~\ref{thm:path}(i) where we would expect Client to be able to achieve a path of length $ 2^k $ while we prove that he cannot hope to get more than $ 2k $.\\

%Bednarska-Bzd\c{e}ga et al. showed in \cite{PC_H} that the critical bias for the Waiter-Client $ k $-clique game (i.e. Waiter is trying to force Client to build a copy of $ K_k $) is $ \Theta(n^{2/(k-1)})=\Theta(n^{1/m(K_k)}) $, and conjecture that for any graph $ H $ the critical bias for $ WC(K_n, H, q) $ is $ \Theta(n^{1/m(H)}) $. Notice that this is in compliance with the random intuition as the probability threshold for $ G_{n,p} $ to contain a copy of $ H $ is $ n^{-1/m(H)} $.

In \cite{PC_H} Bednarska-Bzd\c{e}ga, Hefetz, and \L{}uczak showed that for any tree $ T_k $ on $ k $ vertices the critical bias for $ WC(K_n, T_k, q) $ is $ \Theta (n^{k/(k-1)}) $ and that the critical bias for $ WC(K_n,K_k, q) $ is $ \Theta(n^{2/(k-1)}) $ and conjectured that for any graph $ H $ the critical bias for $ WC(K_n, H, q) $ is $ \Theta(n^{1/m(H)}) $. Notice that this is in compliance with the random intuition as the probability threshold for $ G_{n,p} $ to contain a copy of $ H $ is $ n^{-1/m(H)} $.
On the other hand, in the Client-Waiter game with bias of order $ n^{k/(k-1)} $ Client will not be able to build a star graph with three vertices (by Proposition~\ref{prp:star}) and he will not be able to build a path of length $ 3\log_2 k $ (by Theorem~\ref{thm:path}(i)). This shows that in these games Client is also weak compared to Waiter in the corresponding Waiter-Client game. It is plausible that the upper bound in Theorem~\ref{thm:H-game on K_n} can be improved to $ \Theta (n^{1/m(H)}) $ as well.\\

%\hl{add reference to the MB H- game, biased positional games for which random strategies are nearly optimal}
We mention here three results of similar games for the Maker-Breaker variation, though their bounds are not comparable to ours. The first is by Beck (\cite{BECK97}), it states that Maker can build a cycle of length at least $ (1-e^{-1/200\epsilon})n $ when playing on $ K_n $ against Breaker with a bias of $ \epsilon n $ (while on the other hand Breaker can isolate at least $ \tfrac{\epsilon}{2}e^{-1/\epsilon}n $ vertices in Maker's graph). \\
The second is due to Bednarska-Bzd\c{e}ga and \L{}uczak (\cite{BL01}), who showed in particular that when $ q= (1+\epsilon)n $ then Breaker can prevent a component of size larger than $ 1/\epsilon $, while if $ q=(1-\epsilon)n $ then Maker can build a component of size at least $ \epsilon n $. \\
Yet in another paper of Bednarska-Bzd\c{e}ga and \L{}uczak (\cite{BiasedMB}), they proved that the critical bias for the Maker-Breaker $ H $-game is $ \Theta(n^{1/m_2(H)}) $.

\subsection{Preliminaries}
For the sake of simplicity and clarity of presentation, we do not make a particular effort to optimize
the constants obtained in some of our proofs. Our graph-theoretic notation is standard and follows those in \cite{W01}. In particular, we use the following.

For a graph $ G $, let $ V(G) $ and $ E(G) $ denote its sets of vertices and edges, respectively, and let $ v_G =|V(G)| $ and $ e_G = |E(G)| $. For two sets $ A,B \subseteq V (G) $, let $ E_G(A) $ denote the set of edges of $ G $ with both endpoints in $ A $ and let $ e_G(A) = |E_G(A)| $. Let $ E_G(A,B) $ denote the set of edges of $ G $ with one endpoint in $ A $ and the other endpoint in $ B $ (formally, $ E_G(A,B)=\{e\in E(A\cup B): e\cap A\neq\emptyset, e\cap B\neq\emptyset\}  $), and let $ e_G(A,B) = |E_G(A,B)| $. Notice that if $ A\subseteq B $ then $ E_G(A,B)=E_G(A)+E_G(A,B\backslash A) $. For a set $ S \subseteq V (G)$, let $  G[S] $ denote the subgraph of $ G $ induced by the set $ S $, and $ N_G(S) = \{v \in V (G) \backslash S: \exists u \in S \text{ such that } (uv) \in E(G)\}$ denotes the external neighbourhood of $ S $ in $ G $. For a vertex $ u \in V (G) $ we abbreviate $ N_G(\{u\}) $ under $ N_G(u) $ and let $ d_G(u) = |N_G(u)| $ denote the degree of $ u $ in $ G $. Often, when there is no risk of confusion, we omit the subscript $ G $ from the notation above. The \emph{maximum degree} of a graph $ G $ is $ \Delta(G) := \max\{d_G(u) : u \in V(G)\} $ and the \emph{minimum degree } of a graph $ G $ is $ \delta(G) := \min\{d_G(u) : u \in V (G)\} $.

For a family $ \cl{F} $ of subsets of $ X $, we define the \emph{transversal family} of $ \cl{F} $ to be $ \cl{F}^* = \{A \subseteq X : A \cap B\neq\emptyset \text{ for every } B \in \cl{F}\} $.

Assume that some Client-Waiter game, played on the edge set graph $ G = (V,E) $, is in progress. At any given moment during this game, let $ E_C, E_W, E_F $ denote the set of edges that were claimed by Client, resp. Waiter, resp. unclaimed (free) up to that moment. We denote their respective sizes by $ e_C=|E_C| $, $ e_W=|E_W| $, $ e_F=|E_F| $. If $ A,B\subseteq V $ are two sets then $ E_C(A) $ is the set of Client's edges inside $ A $, $ E_C(A,B) $ is the set of Client's edges with one end in $ A $ and the other in $ B $, and $ e_C(A)=|E_C(A)| $ and $ e_C(A,B)=|E_C(A,B)| $. Similarly we define $ E_W(A), e_W(A), E_F(A) $, etc.

The rest of this paper is organized as follows: in Section~\ref{sec:tools} we quote two useful criteria for Client's win in Client-Waiter and Waiter-Client games, and we state and prove a result of our own which is of independent interest. In Section~\ref{sec:games_on_K_n} we discuss games on the complete graphs and prove Proposition~\ref{prp:star}, Theorems~\ref{thm:component} and~\ref{thm:path}, and Proposition~\ref{prp:k-clique lower bound}. In Section~\ref{sec:H-game_on_G_n,p} we discuss games on $ G_{n,p} $ and give the proofs of Theorem~\ref{thm:CW_random_H_biased} and Proposition~\ref{prp:CW_random_tree}. Section~\ref{sec:discussion} is devoted to concluding remarks.

\section{Game-Theoretic Tools}\label{sec:tools}
In this section we present several general criteria for the existence of a winning strategy for Client in a Client-Waiter or Waiter-Client game.
\begin{thm}[implicit in \cite{BECK08}]\label{thm:tool_ES}
	Let $ q $ be a positive integer, let $ X $ be a finite set, let $ \mathcal{F} $ be a family of subsets of $ X $. If
	\[ \sum_{A\in\mathcal{F}}(q+1)^{-|A|}<1 \]
	then Client has a winning strategy in the $ WC(X,\cl{F},q) $ game.
\end{thm}

Theorem~\ref{thm:tool_ES} gives a criterion for Client to \emph{avoid} the family $ \cl{F} $, which can sometimes be helpful in showing that his claimed subset has some desirable property (see for example the proofs of Theorems~\ref{thm:component}(iii) and \ref{thm:path}(iii) in Section \ref{proof: path_comp_lb}). However, recall that in the Client-Waiter game we introduced a monotonicity rule which allows Waiter to offer less than $ q+1 $ elements in a turn. This difference between the games prevents us from using Theorem~\ref{thm:tool_ES}. The next theorem provides a workaround.
\begin{thm}\label{thm:tool_ES_subset}
	Let $ q $ be a positive integer, let $ X $ be a finite set, let $ \mathcal{F} $ be a family of subsets of $ X $ and let $ \Phi(\mathcal{F})=\sum_{A\in\mathcal{F}} (q+1)^{-|A|} $. Then, playing a Client-Waiter game on $ X $ with bias $ q $, Client has a strategy to claim a set $ X_C\subseteq X $  of size $ |X_C|\geq \lfloor |X|/(q+1)\rfloor $ which fully contains at most $ 2\Phi(\cl{F}) $ sets of $ \mathcal{F} $.
\end{thm}
\begin{proof}
	Denote by $ W_i $ the set of elements offered by Waiter at the $ i $-th turn, and let $ \alpha_i=|W_i|/({q+1}) $.  Suppose Client plays the following random strategy: he picks an element from $ W_i $ uniformly at random, and then with probability $ \alpha_i $ puts it in $ X_C $. \\
	If element $ x\in X $ is offered in the $ i $-th turn, then $ \Pr[x\in X_C]=\alpha_i/|W_i|=1/{(q+1)}. $ Fix some $ A\in\cl{F} $. If, in some turn, Waiter offered at least two elements of $ A $, then surely Client will not fully claim $ A $. Since any element must be offered at some point we get that $ \Pr[A\subseteq X_C]\leq (q+1)^{-|A|} $, and thus $ \mathbb{E}(|\{A\in\mathcal{F}:A\subseteq X_C\}|)\leq\Phi(\mathcal{F}) $. It follows by Markov's inequality that
	\begin{align}
	\Pr[|\{A\in\mathcal{F}:A\subseteq X_C\}|>2\Phi]<1/2.\label{eq:prob 1}
	\end{align}
	Let $ m $ denote the total number of turns played in the game. Note that $ |X_C|=\sum_{i=1}^{m}Z_i $ where $ Z_1,\ldots, Z_m $ are independent Bernoulli random variables with $ \Pr[Z_i=1]= \alpha_i $. Hence
	\begin{align}
	\Pr[|X_C|\geq \lfloor|X|/(q+1)\rfloor]\geq \Pr[\text{Bin}(|X|,1/(q+1))\geq \lfloor|X|/(q+1)\rfloor ]\geq 1/2,\label{eq:prob 2}
	\end{align}
	where the first inequality holds by Theorem 5 from ~\cite{BinPos}.
	Combining (\ref{eq:prob 1}) and (\ref{eq:prob 2}) we conclude that with positive probability both $ |X_C|\geq \lfloor|X|/(q+1)\rfloor $ and $ |\{A\in\mathcal{F}:A\subseteq X_C\}|\leq 2\Phi $, and therefore there is a strategy for Client which will ensure a subset with these properties.
\end{proof}

For a set $ X $ and a family of subsets $ \cl{F} $ we defined $ \cl{F}^* $ to be the transversal family of $ \cl{F} $. If Client wins the $ CW(X,\cl{F}^*,q) $ game, then he has claimed at least one element from every set of $ \cl{F} $. The next theorem is therefore very useful in those situations where we want to show that Client can prevent Waiter from fully claiming a set of $ \cl{F} $.
\begin{thm}[implicit in Theorem 3.2 of \cite{WC_CW_ham}]\label{thm:tool_transversal}
	Let $ q $ be a positive integer, let $ X $ be a finite set and let $ \cl{F} $ be a family of subsets of $ X $. If
	\[ \sum_{A\in\cl{F}} e^{-|A|/(q+1)}<1, \]
	then Client has a winning strategy for the $ CW(X,\cl{F}^*,q) $ game.
\end{thm}

\section{Games  on $ K_n $}\label{sec:games_on_K_n}
\subsection{Star game}
\begin{proof}[\bfseries Proof of Proposition~\ref{prp:star}]
	First notice that the lower bound is trivial since if $ q+1< (n-1)/(k-1) $ then at the end of the game $ e_C> n(k-1)/2 $ which means that the average degree in Client's graph will be higher than $ k-1 $. We turn to the upper bound.\\
	Consider the following strategy for Waiter. In the first turn Waiter chooses some vertex $ v_0 $ and offers $ q+1 $ edges incident to that vertex. Suppose Client picks $ (v_0,v_1) $. In the second turn Waiter will offer some $ q+1 $ arbitrary edges incident to $ v_1 $. In general, if on the $ i $-th turn Client chose the edge $ (v_{i-1},v_i) $, and there are free edges incident to $ v_i $, then in the next turn Waiter will offer an arbitrary subset of those edges of size $ q+1 $, or all of them if there are less then $ q+1 $. If there are no free edges incident to $ v_i $ then Waiter will choose some other vertex with free edges incident to it, and offer a subset of size $ q+1 $ (or all) of those edges. Consider some vertex $ v $. According to the above strategy, we can pair Client's edges which are incident to $ v $ (with exception, perhaps, of the last edge) such that each pair was claimed in consecutive turns, and when Client claimed the later of the two Waiter claimed $ q $ edges incident to $ v $ (again with exception of the last edge). This observation leads to the upper bound on Client's maximum degree: $  \Delta_C\leq 2\lceil (n-1)/(q+2)\rceil $. Assuming $ q\geq\lceil n/k\rceil-2 $ we get
\[ 	\Delta_C\leq 2\left\lceil \dfrac{n-1}{q+2}\right \rceil\leq 2\left\lceil\dfrac{k(n-1)}{n}\right\rceil=2k.  \]

\end{proof}

\subsection{Large component and long path games}
\subsubsection{Waiter's strategies}
We start by presenting Waiter's strategies in these games. These strategies will be promptly analysed to get the upper bounds of Theorems \ref{thm:component} and \ref{thm:path}.\\

\underline{\textbf{Strategy $ S_C $ for Waiter in the Component game}}\\
Let $ n,q $ be integers with $ q+1\geq n-1 $. We describe strategy $ S_C $ for Waiter in a Client-Waiter game on $ E(K_n) $ with bias $ q $. Waiter will maintain 3 subsets $ X,Y,U\subseteq V $. Initiate $ X=Y=\emptyset;\;U=V $. We describe the strategy in 3 stages.\\

\textbf{Stage I}\\
In the $ i $-th turn Waiter picks a maximal set of vertices $ T_i\subseteq U $ such that $ e_F(T_i,U)\leq q+1 $. Denote $ t_i=|T_i| $, to be used later in the proof. \\
Waiter will offer all edges $ E_F(T_i,U) $. Suppose Client picks an edge $ (x,y) $ with $ x\in T_i $, then we add $ x $ to $ X $ and $ y $ to $ Y $, and remove $ T_i\cup\{y\} $ from $ U $. \\
We enumerate the vertices in $ X=\{x_1,x_2,...\},Y=\{y_1,y_2,...\} $ by the order of their addition. It is not hard to verify the following properties right after the $ i $-th turn:
\begin{enumerate}
	\item Client's graph is a perfect matching between $ X $ and $ Y $. In particular $ |X|=|Y|=i $.
	\item For any $ 1\leq j\leq i $, $ x_j $ has only free edges to (some or all) $ y_k $ with $ k<j $.
%	\item All the edges inside $ U  $ are free.
\end{enumerate}
This stage ends when $ U=\emptyset $. We denote by $ s $ the number of turns played in the first stage.\\

\textbf{Stage II}\\
This stage will last for at most $ s $ turns. In the $ i $-th turn Waiter will offer all free edges between $ y_i $ and all isolated vertices in Client's graph. By the end of this stage Client's graph is made of $ s $ components, each of size at most 3. We claim that between any pair of components there are at most 3 free edges, and those are all the remaining free edges in the game. \\
Indeed, let $ C_i, C_j $ be two components. We may assume that $ C_i=\{x_i,y_i,z_i\}, C_j=\{x_j,y_j,z_j\} $ where $ (y_i,z_i) $ and $ (y_j,z_j) $ are the edges claimed in this stage, and that $ i<j $. By property (2) of the previous stage, $ x_i $ does not have a free edge to any vertex in $ C_j $ and $ (x_j,z_i) $ is not available as well. By Waiter's play on this stage the edge $ (y_i,z_j) $ is not free and the edge $ (z_i,z_j) $ was offered sometime during the previous stage, so the free edges between $ C_i$ and $ C_j $ satisfy $ E_F(C_i,C_j)\subseteq \{(y_i,y_j), (y_i,x_j), (z_i,y_j)\} $.\\

\textbf{Stage III}\\
In the last stage Waiter creates an auxiliary board of $ K_s $ and identifies each vertex of this board with one of Client's components of the original game. Given Waiter's strategy for the Client-Waiter game on $ E(K_s) $ with a bias of $ \lfloor q/3\rfloor $, he can use this strategy to play in the original game by offering all edges between two components each time this strategy requires him to offer the edge between the corresponding vertices on the auxiliary board, and if Client chooses some edge which connects two of his components on the original board then Waiter will translate it to the appropriate edge between the corresponding vertices in the auxiliary board. This gives us the recursion  
\begin{align}
\mathcal{C}(n,q)\leq 3\cdot\mathcal{C}(s,\lfloor q/3\rfloor).\label{eq:comp_recursion}
\end{align}

\underline{\textbf{Strategy $ S_P $ for Waiter in the Path game}}\\
We keep all the notations of the previous strategy. This strategy is very similar, we note only the differences. Waiter plays in 2 stages.\\

\textbf{Stage I}\\
In the $ i $-th turn Waiter will pick a maximal set of vertices $ T_i\subseteq U $ such that $ e_F(T_i,V)\leq q+1 $. Denote: $ t_i=|T_i| $. \\
Waiter will offer all edges $ E_F(T_i,V) $. Suppose Client picks an edge $ (x,y) $ with $ x\in T_i $, then we add $ x $ to $ X $ and $ y $ to $ Y $, and remove $ T_i\cup\{y\} $ from $ U $. \\
We will have the following properties after the $ i $th turn:
\begin{enumerate}
	\item Client's graph is a union of at most $ i $ disjoint stars with all the center vertices in $ Y $ and the leaves in $ X $.
	\item All the free edges are in $ Y\cup U $.
\end{enumerate}
This stage will end when $ U=\emptyset $. By the end of this stage the vertices in $ Y $ are centers of disjoint stars in Client's graph, and the edges inside $ Y $ are all the remaining free edges in the game. We denote $ s=|Y| $. Clearly, $ s $ is at most the number of turns played at the first stage.\\

\textbf{Stage II}\\
As in the last stage of strategy $ S_C $ Waiter translates the game to an auxiliary game on $ E(K_s) $, but this time with the same bias $ q $. This leads to the following recursion
\begin{align}
\mathcal{P}(n,q)\leq \mathcal{P}(s,q)+2. \label{eq:path_recursion}
\end{align}

Strategies $ S_C $ and $ S_P $ will be used twice each in the proofs of (i) and (ii) of Theorem~\ref{thm:component} and Theorem~\ref{thm:path}, respectively. The difference between proving (i) and (ii) will be in the degree of precision required in the analysis of the strategies.
\subsubsection{$ k $-sized component and $ k $-path}
\begin{proof}[\bfseries Proof of Theorem~\ref{thm:component}(i)]
When $ q\geq n^2/2 $, Client's graph will be a single edge, so the theorem is true for $ k=1 $. We assume $ q<n^2/2 $ and proceed by induction on $ k $.\\
Let $ n $ be arbitrary and set $ q=6n^{2^k/(2^k-1)} $. 
Waiter will play strategy $ S_C $. Recall that for each $ 1\leq i\leq s $, $ t_i $ is the size of a largest subset $ T_i\subset U $ such that Waiter can offer all edges $ E_F(T_i,U) $. Certainly $ t_i\geq \lfloor q/n\rfloor $ for any $ i $. So the number $ s $ of turns at the first stage is at most
	\begin{align*}
	&s\leq \left \lceil\dfrac{n}{\lfloor q/n\rfloor}\right \rceil\leq\dfrac{n^2}{q-n}+1\leq \dfrac{n^2}{q-q/6}+1=\dfrac{1}{5}\cdot\dfrac{n^2}{q/6} +1< \dfrac{17}{60}\cdot\dfrac{n^2}{q/6}<\dfrac{1}{3}\left(\dfrac{q}{6}\right) ^{(2^{k-1}-1)/(2^{k-1})}\\
	&\Longrightarrow \left\lfloor\dfrac{q}{3}\right \rfloor \geq 6 s^{2^{k-1}/(2^{k-1}-1)}.
	\end{align*}
	We can therefore use the induction hypothesis with recursion (\ref{eq:comp_recursion}) and get that
	\begin{align*}
	\mathcal{C}(n,q)\leq 3\cdot\mathcal{C}(s,\lfloor q/3\rfloor)<3^{k}.
	\end{align*}
\end{proof}	 

\begin{proof}[\bfseries Proof of Theorem~\ref{thm:path}(i)]
	The proof is very similar to the above proof of Theorem~\ref{thm:component}(i). One easily checks that the claim is true for $ k=1 $. In the induction step we take $ q=3n^{2^k/(2^k-1)} $, and get
	\begin{align*}
	&s\leq \left \lceil\dfrac{n}{\lfloor q/n\rfloor}\right \rceil\leq\dfrac{n^2}{q-n}+1\leq \dfrac{n^2}{q-q/3}+1=\dfrac{1}{2}\cdot\dfrac{n^2}{q/3} +1< \dfrac{2}{3}\cdot\dfrac{n^2}{q/3}=\dfrac{2}{3}\left(\dfrac{q}{3}\right) ^{(2^{k-1}-1)/(2^{k-1})}\\
	&\Longrightarrow q \geq 3 s^{2^{k-1}/(2^{k-1}-1)},
	\end{align*}
	which together with recursion (\ref{eq:path_recursion}) implies
	\begin{align*}
	\mathcal{P}(n,q)\leq \mathcal{P}(s,q)+2<2k.
	\end{align*}
\end{proof}	 

\subsubsection{Polylogarithmic component and polylogarithmic path}
In order to get the second part of Theorem~\ref{thm:component} we need a finer analysis of Waiter's strategy. This is done in the next two lemmas.
\begin{lem}\label{lem:comp_log1}
	For any $ \gamma>\log_23 $ Waiter can prevent a component of size $ 2\left(\dfrac{\ln n}{\ln (q/n)}\right)^{\gamma} $ in Client's graph when playing a Client-Waiter game on $ E(K_n) $ with bias $ q\geq\lceil 6^{1/(2-3^{1/\gamma})} \rceil n $.
\end{lem}
\textbf{Remark}: notice that when $ q=6n^{2^k/(2^k-1)} $, then for large $ n,k $ this lemma roughly gives us $ \cl{C}(n,q)<2\cdot 3^k $ which is only slightly worse than the statement in Theorem~\ref{thm:component}(i).
\begin{proof}
	Fix $ \gamma>\log_23 $ and set $ q=cn $ with $ c=\lceil 6^{1/(2-3^{1/\gamma})} \rceil $. Notice that $ c\geq 7 $. The claim is true for $ n\leq 2c+1 $ since then if $ q= cn $ then $ q\geq \binom{n}{2} $. We proceed by induction on $ n $. Waiter will play strategy $ S_C $. In each turn $ i $ at the first stage we have $ t_i\geq \lfloor q/n\rfloor $, therefore the first stage lasts for at most 
	\[ s\leq \left\lceil \dfrac{n}{\lfloor q/n\rfloor}\right \rceil\leq \dfrac{n}{q/n-1}+1<\dfrac{n^2+q}{q-n}   \]
	turns, which is also an upper bound on the number of components in Client's graph at the end of the second stage. We have
	\begin{align}
	\dfrac{\lfloor q/3\rfloor}{s}&\geq \dfrac{(q/3-1)(q-n)}{n^2+q}\geq\left(\dfrac{q}{n}\right)^2\cdot\dfrac{(1/3-1/q)(1-1/c)}{1+q/n^2}\nonumber\\ 
	&\geq \left(\dfrac{q}{n}\right)^2\cdot\left (\dfrac{1}{3}-\dfrac{1}{100}\right )\cdot\dfrac{6}{7}\cdot\dfrac{2}{3}\geq 0.18\cdot\left(\dfrac{q}{n}\right)^2,\label{eq:comp_log_ub1}
	\end{align}
	(we used the assumptions that $ 100< 2c(c+1)\leq q\leq n^2/2 $). We get that
	\[ \dfrac{\lfloor q/3\rfloor}{s}\geq 0.18 \cdot \left(\dfrac{q}{n}\right) ^2\geq \dfrac{c^{2-3^{1/\gamma}}}{6}\cdot\left(\dfrac{q}{n}\right)^{3^{1/\gamma}}>\left(\dfrac{q}{n}\right)^{3^{1/\gamma}}.   \]
	Using recursion (\ref{eq:comp_recursion}) and our induction hypothesis we derive
	\[ \mathcal{C}(n,q)\leq 3\cdot\mathcal{C}(s,\lfloor q/3\rfloor)<6\left(\dfrac{\ln n}{\ln (\lfloor q/3\rfloor/s)} \right)^{\gamma}< 6\left(\dfrac{\ln n}{3^{1/\gamma}\ln(q/n)}\right)^{\gamma}\leq 2\left(\dfrac{\ln n}{\ln(q/n)}\right)^{\gamma}.  \]
\end{proof}

\begin{lem}\label{lem:comp_log2}
	There is $ n_0 $ such that for all integers $ n>n_0 $ and $ q\geq 1.6n $, $ \mathcal{C}(n,q)\leq 3\mathcal{C}(n_1,q_1) $ with $ q_1= \lfloor q/3 \rfloor $ and $ \dfrac{q_1}{n_1}\geq 1.001\dfrac{q}{n} $.
\end{lem}
\begin{proof}
	In order to prove this lemma we need to get a better bound on $ s $ than that we used in Lemma~\ref{lem:comp_log1}. Denote by $ s_k $ the number of moves in the first stage for which $ t_i=k $. For simplicity and clarity of the calculations we make the following sub-optimal assumptions.
	\begin{itemize}
		\item While $ n\geq|U|>q/2 $ we assume $ t_i= 1 $. During this time $ U $ gets decreased by at least 2 vertices a turn, which gives $ s_1\leq (n-q/2)/2 $.
		\item For every $ k\geq 2 $, while $ q/k\geq |U|>q/(k+1) $ we assume $ t_i=k $. During this time $ U $ gets decreased by $ k $ vertices a turn, which gives $ s_k\leq (q/k-q/(k+1))/k=q/k^2(k+1) $.
	\end{itemize}
	This leads to the following bound on $ s $:
	\[ s\leq \sum_{k=1}^{\infty}s_k\leq \dfrac{n-q/2}{2}+\sum_{k=2}^{\infty}\dfrac{q}{k^2(k+1)}\leq \dfrac{n}{2}-\dfrac{q}{4}+\dfrac{q(\pi^2-9)}{6}<0.999\dfrac{n}{3}.\]
	Notice that the necessity to achieve the last inequality is the reason for the constant 1.6 in our upper bound. Set $ n_1=s,\;q_1=\lfloor q/3\rfloor. $ Then 
	\[ q_1\geq \dfrac{q}{3}-1\geq \dfrac{qn_1}{0.999n}-1>1.001\dfrac{q}{n}n_1, \]
	where the last inequality is for large enough $ n $. By our recursion:
	\[ \mathcal{C}(n,q)\leq 3\cdot\mathcal{C}(n_1,q_1). \]
\end{proof}

\begin{proof}[\bfseries Proof of Theorem~\ref{thm:component}(ii)]
	Let $ n $ be an integer. Due to the monotonicity of the game it is enough to prove the Theorem for $ q=1.6n $. Let $ \gamma=\log_23+\epsilon/2 $ and let $ c=c(\gamma) $ be such that Waiter can prevent a component of size $ 2(\ln n/\ln c)^\gamma $ when playing with bias $ cn $ (which we get from Lemma~\ref{lem:comp_log1}). Set $ m=\ln c/\ln1.001 $. Let $ n_1,q_1 $ be the integers guaranteed by Lemma~\ref{lem:comp_log2}. We can recurrently use this lemma to get a sequence $ \{(n_i,q_i)\}_{i=1}^{\ell} $ of pairs with $ \dfrac{q_i}{n_i}\geq 1.001\dfrac{q_{i-1}}{n_{i-1}} $; with $ \ell $ the first such that $ q_\ell\geq cn_\ell $. Notice that necessarily  $ \ell\leq m $ and that since for any $ i $, $ q_i=\lfloor q_{i-1}/3\rfloor $,
	\[ n_{\ell-1}> \dfrac{q_{\ell-1}}{c}\geq \dfrac{q}{c\cdot 4^\ell}= \dfrac{1.6n}{c\cdot 4^{\ell}}, \]
	and so if $ n>n_04^{\ell}c/1.6 $, where $ n_0 $ is from Lemma~\ref{lem:comp_log2}, then  our use of the lemma was valid. Finally we get from Lemma~\ref{lem:comp_log1}
	\[ \mathcal{C}(n,q)\leq 3\cdot\cl{C}(n_1,q_1)\leq\ldots\leq 3^\ell\cdot\mathcal{C}(n_\ell,q_\ell)
	<  2\cdot 3^m\left(\dfrac{\ln n_\ell}{\ln c}\right)^{\gamma}\leq (\ln n)^{\log_23+\epsilon},  \]
	for $ n $ large enough.
\end{proof}

\begin{prop}\label{prop:path_log}
	For all integers $ n,q=cn $ with $ c\geq 2 $ , $ \mathcal{P}(n,q)< 2\log_2\log_{2c/3}q+1 $. 
\end{prop}
\begin{proof}
	Fix $ c $. The claim is true for every $ 1\leq n\leq 2c+1 $ since then $ q\geq\binom{n}{2} $. We proceed by induction on $ n $. Playing strategy $ S_P $ we have that in every turn at the first stage
	\[ t_i\geq \left\lfloor \dfrac{q}{n}\right\rfloor. \]
	In each turn Client chooses an edge $ (x_i,y_i) $ with $ x_i\in T_i $ and $ y_i\in V $.  We consider three cases:
	\begin{itemize}
		\item If $ y_i\in V\backslash U $ then $ y_i\in Y $ and the edge $ (x_i,y_i) $ is just an additional edge to an existing star in Client's graph.
		\item If $ y_i\in T_i $ then $ y_i $ has no more free edges and Waiter can ignore the edge $ (x_i,y_i) $ in the next stage of his strategy.
		\item If $ y_i\in U\backslash T_i $ then Client has just created a new star and $ U $ got decreased by $ t_i+1 $ vertices.
	\end{itemize}
	Therefore, by the end of the first stage Client's graph is a union of at most 
	\[ s\leq \left\lceil\dfrac{n}{\lfloor q/n\rfloor+1}\right\rceil\leq \dfrac{q}{c^2}+1\leq \dfrac{3q}{2c^2} \]
	disjoint stars (we used in the last inequality that $ n> 2c\Longrightarrow q>2c^2 $).\\
	Hence $ q\geq\dfrac{2c^2}{3}\cdot s $, and we get from our recursion and the induction hypothesis that
	\[ \mathcal{P}(n,q)\leq 2+\mathcal{P}(s,q)< 3+2\log_2\log_{(2c/3)^2}q=2\log_2\log_{2c/3}q+1. \]
\end{proof}

\begin{proof}[\bfseries Proof of Theorem~\ref{thm:path}(ii)]
	Let $ n $ be an integer. Due to monotonicity of the game we may assume $ q=n $. When Waiter plays strategy $ S_P $ Client's graph after the first stage contains at most $ s\leq n/2 $ disjoint stars, hence $ q\geq 2s $ and by the recursion and Proposition~\ref{prop:path_log}
	\[ \mathcal{P}(n,q)\leq 2+\mathcal{P}(s,q)<3+2\log_2\log_{4/3}q\leq 3\ln\ln n, \]
	for $ n  $ large enough.	
\end{proof}

\subsubsection{Client's side: linear-sized component and linear-sized path}\label{proof: path_comp_lb}
\begin{proof}[\bfseries Proof of Theorem~\ref{thm:component}(iii)]
	Set $ \delta=\dfrac{\epsilon}{1-\epsilon},\; \theta=e^{-2.5/\delta-1} $ and let 
	\[ \mathcal{F}:=\{E(H): H\subseteq K_n, v_H\leq\theta n, e_H\geq (1+\delta)v_H \}. \]
	Then
	\begin{align*}
	\Phi(\mathcal{F})&:=\sum_{i=4}^{\lfloor\theta n\rfloor}\binom{n}{i}\binom{\binom{i}{2}}{(1+\delta)i}(q+1)^{-(1+\delta)i}< \sum_{i=4}^{\lfloor\theta n\rfloor}\left [\dfrac{en}{i}\left (\dfrac{(1-\epsilon)ei}{2}\right )^{1+\delta}\left (\dfrac{(1-\epsilon)n}{2}\right )^{-(1+\delta)}  \right ]^i\\
	&=\sum_{i=4}^{\lfloor\theta n\rfloor}\left[e^{2+\delta}\left (\dfrac{i}{n}\right )^\delta \right ]^i\leq \sum_{i=4}^{\lfloor\theta n\rfloor} \left ( e^{2+\delta} \theta^\delta\right )^i<\sum_{i=4}^{\infty} e^{-i/2}<\dfrac{1}{2}.
	\end{align*}
	By Theorem~\ref{thm:tool_ES_subset} Client has a strategy such that by the end of the game his graph contains a subgraph $ G_C $, with $ e(G_C)= \dfrac{n(n-1)}{2(q+1)}\geq\dfrac{n}{1-\epsilon} $ edges such that every connected component $ U $ with size at most $ \theta n $ has less than $ (1+\delta)|U| $ edges in $ G_C $. Suppose that all the components in $ G_C $ are of size less than $ \theta n $. Then 
	\[ |E(G_C)|=\sum_{U\in \text{comp}(G_C)}e_{G_C}(U)<(1+\delta)\sum_{U\in \text{comp}(G_C)}|U|\leq(1+\delta)n=\dfrac{n}{1-\epsilon}, \]
	which is a  contradiction. Therefore Client has a connected component of size at least $ \theta n $.
\end{proof}
\begin{proof}[\bfseries Proof of Theorem~\ref{thm:path}(iii)]
\begin{lem}\label{lem:linear_path}
	Let $ \epsilon,\gamma>0 $ and $ \delta_1>\delta_2>0 $ be constants. Let $ G $ be a graph on $ n $ vertices with the following properties:
	\begin{enumerate}
		\item $ G $ has $ e_G\geq (1+\epsilon)n $ edges,
		\item every set $ S $ of size $ |S|\leq \delta_1 n $ spans $ e(S)<(1+\epsilon)|S| $ edges,
		\item every set $ S $ of size $ |S|\leq \delta_2 n $ spans $ e(S)<(1+\epsilon/2)|S| $ edges,
		\item for every set $ S $ of size $ |S|\leq \gamma n $, we have $ e(S,V\backslash S)<\dfrac{\epsilon\delta_2}{2}n $ edges.
	\end{enumerate}
	Then $ G $ contains a path of length at least $ \gamma n $.
\end{lem}
\begin{proof}
	By the first property $ G $ must contain some connected component $ C $ with $ e(C)\geq (1+\epsilon)|C| $. By the second property $ |C|>\delta_1 n $. Consider the DFS algorithm as defined, for example, in \cite{DFS}. As a quick reminder: we take an arbitrary ordering of the vertices of $ G $ and run a DFS exploration on $ G $ by maintaining three sets of vertices: $ S $ --- the vertices we have finished exploring, $ U $ --- a LIFO stack with the vertices we are currently exploring, and $ T $ --- the unvisited vertices. Consider an execution of the DFS algorithm on $ C $, starting with $ S=U=\emptyset $ and $ T=C $, and completing when $ S=C $ and $ U=T=\emptyset $ . We will use the following properties of this algorithm:
	\begin{itemize}
		\item at any given moment there are no edges of $ G $ between $ S $ and $ T $, and
		\item at any given moment $ U $ spans a path in $ C $.
	\end{itemize}
	Set $ t_0=0 $ and $ t_f=\lceil |C|/(\delta_2 n)\rceil $. For any $ 1\leq i\leq t_f-1 $ let $ t_i $ be the moment in which $ |S|=\delta_2ni $ and let $ S_i $, $ 1\leq i\leq t_f $, be the vertices which were added to $ S $ between time $ t_{i-1} $ and $ t_i $. Since $ |S_i|\leq \delta_2 n $ for all $ i $, we get from the third property that 
	\[ \sum_{1\leq i<j\leq t_f} e(S_i,S_j)=e(C)-\sum_{i=1}^{t_f}e(S_i)\geq (1+\epsilon)|C|-(1+\epsilon/2)\sum_{i=1}^{t_f}|S_i|=\dfrac{\epsilon}{2}|C|. \]
	Thus there is $ 1\leq i_0\leq t_f-1 $ such that $ \sum_{i_0<j\leq t_f}e(S_{i_0},S_j)\geq \dfrac{\epsilon|C|}{2(t_f-1)}\geq\dfrac{\epsilon\delta_2 n}{2}  $. Since at time $ t_{i_0} $ all these edges are between $ S_{i_0} $ and $ U $, we get from property 4 that at that moment $ U $ spans a path of length at least $ \gamma n $.
\end{proof}
We now return to the proof of Theorem~\ref{thm:path}(iii). Since the game is bias monotone, we may assume $ q=\dfrac{n}{2(1+\epsilon)}>(1-\epsilon)\dfrac{n}{2} $. Set $ \delta_1=e^{-3/\epsilon-1} $, $ \delta_2=\delta_1^2 $, and $ \gamma=(\epsilon\delta_2)^2 $.
	Define
	\begin{align*}
	\cl{F}_1&:=\{E(H):H\subseteq K_n, v_H\leq \delta_1 n,e_H\geq(1+\epsilon)v_H  \},\\
	\cl{F}_2&:=\{E(H):H\subseteq K_n, v_H\leq \delta_2 n,e_H\geq(1+\epsilon/2)v_H  \},\\
	\cl{F}_3&:=\{E(H): H\subseteq K_n, H=(S\cup(V\backslash S), E), |S|\leq\gamma n,e_H\geq \tfrac{\epsilon\delta_2}{2}n \}.				
	\end{align*}
	We calculate
	\begin{flalign*}
	\Phi(\cl{F}_1)&=\sum_{i=4}^{\delta_1 n}\binom{n}{i}\binom{\binom{i}{2}}{(1+\epsilon)i}(q+1)^{-(1+\epsilon)i}<\sum_{i=4}^{\delta_1 n}\left[\dfrac{en}{i}\left(\dfrac{ei}{2(1+\epsilon)}\right)^{1+\epsilon}\left(\dfrac{n}{2(1+\epsilon)}\right)^{-(1+\epsilon)}  \right]^i\\
	&=\sum_{i=4}^{\delta_1 n}\left[e^{2+\epsilon}\left (\dfrac{i}{n}\right )^\epsilon \right]^i\leq \sum_{i=4}^{\infty}\left(e^{2+\epsilon}\delta_1^\epsilon \right)^i=\sum_{i=4}^{\infty}e^{-i}<0.1;\\
	\Phi(\cl{F}_2)&=\sum_{i=4}^{\delta_2 n}\binom{n}{i}\binom{\binom{i}{2}}{(1+\epsilon/2)i}(q+1)^{-(1+\epsilon/2)i}<\sum_{i=4}^{\delta_2 n}\left[\dfrac{en}{i}\left(\dfrac{ei}{2(1+\epsilon/2)}\right)^{1+\epsilon/2}\left(\dfrac{n}{2(1+\epsilon)}\right)^{-(1+\epsilon/2)}  \right]^i\\
	&=\sum_{i=4}^{\delta_2 n}\left[\left(\dfrac{1+\epsilon}{1+\epsilon/2}\right)^{1+\epsilon/2} e^{2+\epsilon/2}\left (\dfrac{i}{n}\right )^{\epsilon/2} \right]^i\leq\sum_{i=4}^{\delta_2 n}\left[e^{2+\epsilon}\left (\dfrac{i}{n}\right )^{\epsilon/2} \right]^i\leq \sum_{i=4}^{\infty}\left (e^{2+\epsilon}\delta_2^{\epsilon/2}\right )^i\\
	&=\sum_{i=4}^{\infty}e^{-i}<0.1;\\
	\Phi(\cl{F}_3)&=\binom{n}{\gamma n}\binom{\gamma(1-\gamma)n^2}{\tfrac{\epsilon\delta_2}{2}n}(q+1)^{-\tfrac{\epsilon\delta_2}{2}n}\leq \left(\dfrac{e}{\gamma}\right)^{\gamma n}\left(\dfrac{2\gamma(1-\gamma)n}{\epsilon\delta_2}\right)^{\epsilon\delta_2n/2}\left(\dfrac{n}{2(1+\epsilon)}\right)^{-\epsilon\delta_2n/2}\\
	&\leq\left [\left(\dfrac{e}{\epsilon^2\delta_2^2}\right)^{2\epsilon\delta_2}4(1+\epsilon)\epsilon\delta_2\right ]^{\epsilon\delta_2 n/2}=o(1),
	\end{flalign*}	
	where the last equality is for $ \epsilon $ small enough. We get that for $ n $ large enough
	\[ \Phi(\cl{F}_1\cup\cl{F}_2\cup\cl{F}_3)\leq \Phi(\cl{F}_1)+\Phi(\cl{F}_2)+\Phi(\cl{F}_3)<1/2, \]
	and by Theorem~\ref{thm:tool_ES_subset} Client has a strategy to claim a subgraph $ G_C $ which has all the properties of Lemma~\ref{lem:linear_path} and therefore contains a path of length at least $ \gamma n $.
\end{proof}

\begin{proof}[\bfseries Proof of Theorem~\ref{thm:path}(iv)]
	We will use the following lemma.
	\begin{lem}[Lemma 4.4 in \cite{size_ramsey}]\label{lem:size_ramsey}
		Let $ G $ be a graph on $ n $ vertices. Suppose that for every two disjoint sets $ A,B\subset V(G) $ such that $ |A|,|B|\geq k $, there is at least one edge between $ A $ and $ B $. Then $ G $ contains a path of length $ n-2k+1 $.
	\end{lem}
\noindent	Let $ q+1= \epsilon n $, and let $ \delta>0 $ to be determined. Define 
	\[ \mathcal{F}:=\{E_{K_n}(A,B): A\cap B=\emptyset, |A|=|B|=\delta n \}. \]
	If Client wins the $ CW(K_n, \mathcal{F}^*, q) $ game then he has an edge between any two disjoint subsets of size $ \delta n $, and by Lemma~\ref{lem:size_ramsey} he has a path of length $(1-2\delta) n$. So it is enough to verify the condition of Theorem~\ref{thm:tool_transversal}:
%	\begin{align*}
%	&\sum_{A\in\mathcal{F}}e^{-|A|/(q+1)}=\binom{n}{\delta n}\dfrac{1}{2}\binom{(1-\delta)n}{(1-\delta)n/2}e^{-(1-\delta)^2n^2/4(q+1)}\\
%	&\leq \left(\dfrac{e}{\delta}\right)^{\delta n}2^{(1-\delta)n}e^{-(1-\delta)^2n/4\epsilon}=\left [e^{\ln2+\delta(1-\ln2+\ln(1/\delta))-(1-\delta)^2/4\epsilon}\right ]^n,
%	\end{align*}
	\begin{align*}
	\sum_{A\in\mathcal{F}}e^{-|A|/(q+1)}\leq\binom{n}{\delta n}^2e^{-\delta^2n^2/(q+1)}
	\leq \left[\left(\dfrac{e}{\delta}\right)^2e^{-\delta/\epsilon}  \right]^{\delta n}.
	\end{align*}
	The last expression will be asymptotically small when
	\[ \dfrac{\delta}{\epsilon}>2(1+\ln(1/\delta)), \]
%	\[ \dfrac{(1-\delta)^2}{4\epsilon}>\ln2+\delta(1-\ln2+\ln(1/\delta)). \]
	and this is true for $ \delta=2\epsilon\ln(1/\epsilon) $ and $ \epsilon $ small enough.
\end{proof}

\subsection{The $ H $-game}
\begin{proof}[\bfseries Proof of Proposition~\ref{prp:k-clique lower bound}]
The idea of the proof (suggested by Bednarska-Bzd\c{e}ga) is to use the following theorem of hypergraph containers.
\begin{thm}[implicit in Theorem 2.3 in \cite{h_containers}]\label{thm:hypergraphs}
	Let $ H $ be a graph with at least three vertices. Then there are $ n_0,\delta>0 $ such that for every $ n\geq n_0 $ there is a collection $ \cl{C} $ of subgraphs of $ K_n $ such that 
	\begin{enumerate}
		\item Every $ H $-fee subgraph of $ K_n $ is contained in some $ C\in\cl{C} $.
		\item For every $ C\in\cl{C} $, $ e_C\leq(1-\delta)\binom{n}{2} $.
		\item $ |\cl{C}|\leq n^{n_0n^{2-1/m_2(H)}} $.
	\end{enumerate}
\end{thm}
The lower bound for the $ CW(K_n,H, q) $ game is an easy application of the above theorem and of the criterion for Client's win in Theorem~\ref{thm:tool_transversal}.
Let $ \cl{C} $ be the family of containers for the graph $ H $. It is enough to show that Client can claim at least one edge in every complement graph of $ C\in\cl{C} $. Since the number of edges in any such complement is at least $ \delta\binom{n}{2} $, recalling the bound on $ |\cl{C}| $ we can verify the condition of Theorem~\ref{thm:tool_transversal}.
		\begin{align*}
		\sum_{A\in\cl{F}} e^{-|A|/(q+1)}\leq n^{n_0n^{2-1/m_2(H)}}e^{-\delta\binom{n}{2}/(q+1)}=o(1),
		\end{align*}
		provided $ q\leq cn^{1/m_2(H)}/\ln n $ for some $ c=c(H)>0 $.
\end{proof}
	
\section{The H-Game on $ G_{n,p} $}\label{sec:H-game_on_G_n,p}

We note two well known facts about random graphs which will be used in this section without reference. 
\begin{itemize}
	\item When $ n^{-k/(k-1)}\ll p\ll n^{-(k+1)/k} $ w.h.p. $ G_{n,p} $ is a forest with copies of all trees with at most $ k $ vertices, and no tree with more than $ k $ vertices.
	\item Let $ \alpha, c $ be positive constants. If $ p\leq cn^{-1/\alpha} $ then for any fixed graph $ G $ with $ m(G)>\alpha $ w.h.p. $ G\nsubseteq G_{n,p} $. Put another way, for $ c,\alpha $ and $ p $ as above and any fixed $ k>0 $, w.h.p. any subgraph $ G $ of $ G_{n,p} $ on at most $ k $ vertices has density $ m(G)\leq\alpha $.
\end{itemize}

\subsection{Client's side}
The proof of Maker's side (the 1-statement) in \cite{MBH} relayed on hypergraph containers as an auxiliary tool. That method would have worked here as well. However, we give an alternative proof using another tool --- a variant of the famous K\L{}R conjecture which was proved in (\cite{KLR}). We start with a few definitions.
\begin{defn}
	A bipartite graph between sets $ U $ and $ V $ is $ (\epsilon,d) $\emph{-lower-regular} if, for every $ U'\subseteq U $ and $ V'\subseteq V $ with $ |U'|\geq\epsilon |U| $ and $ |V'|\geq\epsilon |V| $, the density $ d(U',V') $ of edges between $ U' $ and $ V' $ satisfies $ d(U',V')\geq d $.
\end{defn}

Given a graph $ H $ with vertex set $ [k] $ , we denote by $ \cl{G}(H,n,d,\epsilon) $ the collection of all graphs $ G $ with vertex set $ V_1\cup\ldots\cup V_k $, where $ V_1,\ldots,V_k $ are pairwise disjoint sets of size $ n $ each, whose edge sets consists of $ e_H $ different $ (\epsilon,d) $-lower-regular bipartite graphs, one graph between $ V_i $ and $ V_j $ for each $ ij\in E(H) $.\\
For an arbitrary graph $ G $ and $ p\in [0,1] $, we denote by $ G_p $ the random subgraph of $ G $, where each edge of $ G $ is included with probability $ p $ independently of all other edges.

\begin{thm}[implied by Theorem 2.1 in \cite{KLR}]\label{thm:2.1}
	Let $ H $ be an arbitrary graph. For every $ d>0 $, there exist $ \epsilon,C>0 $ such that if $ p\geq Cn^{-1/m_2(H)} $, then the following holds. For every $ G\in\cl{G}(H,n,d,\epsilon) $, w.h.p. the random graph $ G_p $ has the following property: Every subgraph $ G' $ of $ G_p $ in $ \cl{G}(H,n,dp,\epsilon) $ contains a copy of $ H $.
\end{thm}

We are ready to prove Client's side in Theorem~\ref{thm:CW_random_H_biased}.
\begin{proof}[\bfseries Proof of the 1-statement of Theorem~\ref{thm:CW_random_H_biased}]
	Choose $ d $ such that $ d<1/(q+1) $ and
	\[ \dfrac{d}{1/2-d}(1+\ln 2 -\ln d)<\dfrac{1}{q+1}	 \]
	(this is possible since the LHS goes to 0 when $ d $ goes to 0), and let $ \epsilon=\epsilon(H,d) $ be that of Theorem~\ref{thm:2.1}. Let $ k=|V(H)| $. At the beginning of the game Client will fix an equipartition of the vertices of $ G_{n,p} $ to $ k $ parts $ V_1,\ldots,V_k $. He will then follow a strategy which guarantees that by the end of the game his graph, $ G_C $, when restricted to any pair of parts, is $ (\epsilon,dp) $-lower-regular. To see that this means that w.h.p. he will claim a copy of $ H $, take $ G $ in Theorem~\ref{thm:2.1} to be the complete $ k $-partite graph on $ V_1,\ldots,V_k $. Clearly $ G\in\cl{G}(H,n/k,d,\epsilon) $ and $ G_C\cap G\in \cl{G}(H,n/k,dp,\epsilon) $, and we can consider $ G_p $ as $ G\cap G_{n,p} $.  It remains to show that Client indeed has such a strategy. To this end we will use Theorem~\ref{thm:tool_transversal}, and define
	\[ \cl{F}:=\{F\subseteq E_{G_{n,p}}(U_1,U_2): U_1\cap U_2=\emptyset, |U_1|=|U_2|=\epsilon n/k, |F|\geq e_{G_{n,p}}(U_1,U_2)-dp\epsilon^2n^2/k^2  \} .\]
	Clearly, if Client wins the $ CW(G_{n,p},\cl{F}^*, q) $ game then he has achieved his goal. It remains to verify the condition of Theorem~\ref{thm:tool_transversal}. Indeed, since w.h.p. the number of edges between any two disjoint subsets of size $ \epsilon n/k $ will satisfy
	\[ \tfrac{1}{2}p\epsilon^2 n^2/k^2\leq e_{G_{n,p}}(U_1,U_2)\leq 2p\epsilon^2 n^2/k^2, \]
	we get
	\begin{align*}
	\Phi(\cl{F})&\leq \binom{n}{2\epsilon n/k}\binom{2\epsilon n/k}{\epsilon n/k}\binom{2p\epsilon^2 n^2/k^2 }{dp\epsilon^2n^2/k^2}e^{-(\tfrac{1}{2}-d)p\epsilon^2n^2/(q+1)k^2}\\
	&\leq \left(\dfrac{ek}{2\epsilon}\right)^{2\epsilon n/k} 2^{2\epsilon n/k} \left (\dfrac{2e}{d}\right )^{dp\epsilon^2n^2/k^2}e^{-(\tfrac{1}{2}-d)p\epsilon^2n^2/(q+1)k^2}\\
	&\leq C^n_{\epsilon,k}\exp\left (\dfrac{p\epsilon^2n^2}{k^2}(d(1+\ln 2 -\ln d)-(1/2-d)/(q+1))\right )\to 0,
	\end{align*}
	by our assumption on $ d $.
\end{proof}

\subsection{Waiter's side}
We start with the case of a graph $ H $ for which there exists $ H'\subseteq H $ such that $ d_2(H')=m_2(H) $, $ H' $ is strictly 2-balanced and it is not a tree or a triangle. Due to the monotonicity of the Client-Waiter game, it is enough to consider the unbiased ($ q=1 $) case. {Moreover, it is enough to show that Waiter can prevent Client from claiming a copy of $ H' $, and so we may assume that $ H'=H $.  Our proof follows very closely that of Theorem 2 in \cite{MBH}. We start with a general sufficient condition for Waiter's win.
\begin{prop}\label{prp:random_waiter_criterion}
	Let $ H $ be a strictly 2-balanced graph which is neither a tree nor a triangle. If $ G $ is a graph such that $ m(G)\leq m_2(H) $, then Waiter has a winning strategy for the $ CW(G,H,1) $ game.
\end{prop}
\begin{proof}
	The \emph{arboricity} of a graph $ G $ is defined by
	\[ ar(G)=\max\limits_{G'\subseteq G}\dfrac{e(G')}{v(G')-1}. \]
	The Nash-Williams arboricity theorem (\cite{NW}) states that any graph $ G $ can be decomposed into $ \lceil ar(G)\rceil  $ edge-disjoint forests.
\begin{lem}\label{lem:arb}
	Let $ G,H $ be graphs such that 
	\[ \left\lceil\dfrac{ar(G)}{2} \right\rceil<ar(H),  \]
	then Waiter has a winning strategy for the $ CW(G,H,1) $ game.
\end{lem}
\begin{proof}
	Set $ k=\left\lceil\dfrac{ar(G)}{2} \right\rceil $, and partition $ E(G) $ into $ 2k $ edge-disjoint forests. Divide these forests to pairs. By Theorem 2 in \cite{CP_Games} Waiter can force Client's graph to be a forest when playing on the edges of a union of two edge-disjoint forests. Thus, when playing on $ G $ Waiter has a strategy to force Client's graph to be a union of $ k $ edge disjoint forests. For any subset $ S\subseteq V(G) $, the number of Client's edges spanned by $ S $ will be at most $ k(|S|-1) $, which means that Client's graph has arboricity at most $ k $, hence it cannot contain $ H $.
\end{proof}
\begin{lem}\label{lem:dens}
	Let $ G,H $ be graphs such that 
	\[ \left\lceil\dfrac{m(G)}{2} \right\rceil<m(H),  \]
	then Waiter has a winning strategy for the $ CW(G,H,1) $ game.
\end{lem}
\begin{proof}
	We first orient the edges of $ G $ in the following manner. Set $ k=\lceil m(G)\rceil $. Construct a bipartite graph between $ E(G) $ and $ k $ copies of $ V(G) $ and connect each edge to all the copies of the vertices which are incident to it. The graph satisfies Hall's condition with respect to $ E(G) $, hence we have a matching which covers $ E(G) $. For any $ v\in V $ orient the edges of $ v $ such that $ v $ is the source of $ e $ if and only if $ e $ is connected to a copy of $ v $ in this matching. Since there are $ k $ copies of $ v $, its out-degree will be at most $ k $. Now Waiter can then play on each vertex at a time, offering only pairs of edges for which the current vertex is a source. The maximum out-degree in Client's graph will then be at most $ \left\lceil\dfrac{m(G)}{2} \right\rceil $, which means that its maximal density is lower than $ m(H) $, and it certainly does not contain $ H $.
\end{proof}	
To prove the proposition we consider several cases of maximal 2-density of $ H $ and use the two lemmas to show that in any case Waiter has a winning strategy. Since this is practically the same as in Theorem 18 in \cite{MBH} we omit the details. The inquisitive reader can find them in Appendix~\ref{apx:A}.

\end{proof}
Consider the game $ CW(G,H,1) $ on the edges of some arbitrary graph $ G $. Certainly, any edge in $ G $ which does not take part in any copy of $ H $ is irrelevant to the outcome of the game. Moreover, if some copy of $ H $ in $ G $ has two edges neither of which takes part in another copy of $ H $, then Waiter can offer these two edges in a single turn, thus preventing Client from claiming this copy of $ H $ while not risking any other copy. This leads to the following definition.
\begin{defn}
	A $ H $-core of $ G $ is a maximal subgraph $ G'\subseteq G $ such that 
	\begin{itemize}
		\item every edge of $ G' $ is contained in at least one copy of $ H $ in $ G' $, and
		\item every $ H $-copy on $ G' $ has at most one edge which does not take part in another $ H $-copy.
	\end{itemize}
\end{defn}
	The exact construction and a proof that the $ H $-core is unique can be found in \cite{MBH}. By the discussion above Waiter wins the  $ CW(G,H,1) $ game if and only if he wins  $ CW(G',H,1) $ where $ G' $ is the $ H $-core of $ G $. Furthermore, it is enough to show that Waiter has a winning strategy for bi-connected (2-connected) components of $ G' $, since these do not share edges (and in particular they do not share $ H $-copies). 
\begin{lem}[Lemma 23 in \cite{MBH}]
	Let $ H $ be a strictly 2-balanced graph which is not a tree or a triangle. Then there exist constants $ c>0 $ and $ L>0 $ such that w.h.p. every bi-connected component of the $ H $-core of $ G_{n,p} $ has size at most $ L $, provided that $ p\leq cn^{-1/m_2(H)} $.
\end{lem}
We can now finish the proof of the 0-statement of Theorem~\ref{thm:CW_random_H_biased} for $ H\neq K_3$. Set $ G=G_{n,p} $ with $ p=cn^{-1/m_2(H)} $. By the above lemma all the bi-connected components of the $ H $-core of $ G $ will be of size at most $ L $. By a well known property of $ G_{n,p} $ (as was mentioned in the beginning of this section), w.h.p. all the bi-connected components will be with maximal density at most $ m_2(H) $ and by Proposition~\ref{prp:random_waiter_criterion} Waiter has a winning strategy while playing on each bi-connected component, and thus when playing on all of the $ H $-core, and indeed on all of $ G $.\\

Next we turn to the case where $ H=K_3 $ and $ q\geq 2 $. The Client's side was covered in the proof of the 1-statement, it remains to show the following:
\begin{prop}
	There is some constant $ c>0 $ such that w.h.p. Waiter wins the $ CW(G_{n,p},K_3,2) $ game when $ p=cn^{-1/2} $. 
\end{prop}
\begin{proof}
We will prove two lemmas --- the first will show that Client can only win on graphs with maximum density higher than 2, and the second will show that when $ p=cn^{-1/2} $, if Client wins on $ G_{n,p} $ then w.h.p. he wins on some subgraph of bounded order.
\begin{lem}\label{lem: Waiter_winning_riterion}
	Let $ G $ be a graph with $ m(G)\leq 2 $, then Waiter has a winning strategy in the $ CW(G, K_3,2) $ game.
\end{lem}
\begin{proof}
	Suppose to the contrary that $ G $ is a minimal graph such that $ m(G)\leq2 $ and Client wins the game on $ G $. By Theorem~1.3 in \cite{Man_waiter} Waiter can force Client's graph to be acyclic when playing on $ K_6 $ with bias 2. We may therefore assume that $ v_G>6 $. Let $ A\subset V(G) $ be a proper subset, and define $ H=G[A] $ and $ \hat{H}=G[V\backslash A] $. We must have that $ e(A, V\backslash A)\geq 4 $, otherwise Waiter can play on $ H $ and then on $ \hat{H} $ (winning on both by the minimality of $ G $), and then offer all the edges $ E(A,V\backslash A) $ and Client will not claim a triangle. In particular $ \delta(G)\geq 4 $. But since $ m(G)\leq 2 $, it must be that $ m(G)=2 $ and $ G $ is 4-regular. This leads to
	\[ e(A,V\backslash A)=4v_H-2e_H\geq 4\Longrightarrow e_H\leq 2(v_H-1).\tag{*}\label{eq:*} \]
	Let $ v_0 $ be an arbitrary vertex, and let $ N(v_0)=\{v_1,v_2,v_3,v_4\} $. Denote $ H=G[\{v_0\}\cup N(v_0)] $ and $ \hat{H}=G[V\backslash (\{v_0\}\cup N(v_0))] $. We claim that $ G[N(v_0)] $ must be a connected graph, for otherwise we can partition $ N(v_0) $ to two parts, $ A,B $, each with size at most 3, and $ E(A,B)=\emptyset $. Waiter can then play his winning strategy on $ G\backslash\{v_0\}, $ then offer $ E(v_0,A) $, and on the last turn he will offer $ E(v_0,B) $. It is easy to see that in this case Client will not claim a triangle. From this reasoning together with (\ref{eq:*}) we deduce that $  e_H\in\{7,8\} $. We now consider several cases. In each case we show that Waiter, after playing his winning strategies on $ H $ and on $ \hat{H} $, has a strategy to offer the remaining free edges such that Client will not claim a triangle.
	\begin{enumerate}
		\item Suppose $ e_H=8 $ and $ G[N(v_0)] $ is isomorphic to $ C_4 $ ($ G[N(v_0)]\cong C_4 $). Then every vertex in $H $ has at most one edge connecting it to $ \hat{H} $.	Suppose there is $ u\in \hat{H} $ such that $ e(u,H)=4 $. Then $ G[H\cup\{u\}] $ has 12 edges which violates (\ref{eq:*}). This means that Waiter (after having played and won on $ H $ and $ \hat{H} $) can just offer in each turn all the free edges incident to some vertex in $ \hat{H} $. Client will not claim a triangle since he will not have a vertex with degree higher than one in the cut between $ H $ and $ \hat{H} $.
		\item Suppose $ e_H=8 $ and $ G[N(v_0)]\cong K_3+e $. Then there is only one vertex in $ H $ with degree 2, let it be $ v_1 $. Waiter will offers all free edges incident to $ v_1 $, and then the remaining (two) edges. Since the two vertices in $ H $ which are connected to $ v_1 $ have degree 4 in $ H $, Client will not be able to claim a triangle.
		\item Suppose $ e_H=7 $ and $ G[N(v_0)]\cong P_4 $. Suppose the path is $ \{v_1,v_2,v_3,v_4\} $. As in case 1, there is no $ u\in \hat{H} $ such that $ e(u,H)=4 $. We have 3 sub-cases.
		\begin{enumerate}
			\item There is $ u\in \hat{H} $ such that $ u$ is connected to $ v_1,v_2,v_3 $. Then the graph $ G[\{v_0,v_1,v_2,v_3,u\}] $ is isomorphic to $ H $ of case 1. The case where $ u $ is connected to $ v_2,v_3,v_4 $ is treated similarly.
			\item There is $ u\in \hat{H} $ such that $ u $ is connected to $ v_1,v_3,v_4 $. Waiter will offer $ (u,v_3), (u,v_4) $ and the other free edge incident to $ v_4 $, and in the next turn he can safely offer the remaining free edges. The case where $ u $ is connected to $ v_1,v_2,v_4 $ is treated similarly.
			\item We can assume that for any vertex $ u\in \hat{H} $, $ e(u,H)\leq 2 $. If there is $ u\in \hat{H} $ which is connected to both $ v_1,v_2 $ then Waiter will offer these edges and the other free edge of $ v_1 $, otherwise he will offer just the free edge of $ v_1 $. In the next turn he will do the same for $ v_4 $, and in the last turn he can just offer the remaining (if any) free edges.
		\end{enumerate}
		\item Suppose $ e_H=7 $ and $ G[N(v_0)]\cong S_3 $. Let $ v_1 $ be the centre vertex. In the last three turns Waiter can offer in each turn the two free edges of $ v_i $ for $ i=2,3,4 $.
	\end{enumerate}
	We have shown that in every case Waiter has a winning strategy, hence there is no such $ G $.
\end{proof}
Continuing with the proof of the proposition, the next definition and the lemma that follows are influenced by the ideas of Nenadov, Steger and Stojakovi\'{c} in \cite{MBH}, but we need to make some necessary changes, since (as mentioned there) their proof will not go through for $ H=K_3 $ as that would be a contradiction to the result of \cite{MBClique}. \\
Let $ G $ be a graph. An edge in $ G $ is \emph{free} if it does not take part in any triangle, it is \emph{open} if it takes part in precisely one triangle, and it is \emph{half-open} if it takes part in precisely two triangles. Otherwise, it is \emph{closed}. 
\begin{defn}
	A $ K_3 $-core of $ G $ is a maximal subgraph $ G'\subseteq G $ such that 
	\begin{itemize}
		\item there are no free edges in $ G' $,
		\item every triangle in $ G' $ has at most one open edge, and
		\item every half-open edge is in at least one triangle which has no open edges.
	\end{itemize}
\end{defn}
Consider the following process for generating a $ K_3 $-core of graph $ G $. We set $ T $ to be the set of all triangles of $ G $, and define the subgraph $ G_T:=\bigcup_{t\in T}t $. Iteratively we remove from $ T $ all triangles with more than one open edge in $ G_T $, and all pair of triangles which share a half-open edge and both of them have an open edge in $ G_T $, updating $ G_T $ after each step. When the process ends $ G_T $ is a $ K_3 $-core of $ G $.
\begin{clm}
	Let $ G $ be a graph, and let $ \cl{G} $ be the family of bi-connected (2-connected) components of a $ K_3 $-core of $ G $. Suppose that for any $ G'\in\cl{G} $, Waiter has a winning strategy in the $ CW(G',K_3,2) $ game. Then Waiter has a winning strategy in the $ CW(G,K_3,2) $ game.
\end{clm}
\begin{proof}
	Since the bi-connected components are pairwise edge-disjoint, a winning strategy for each separate component yields a winning strategy for the $ K_3 $-core. After winning on the $ K_3 $-core Waiter will consider the removed triangles in the $ K_3 $-core generating process described above, but in reverse order. Each time he will add a triangle which has more than one open edge he will offer Client two of those open edges, and each time he adds a pair of triangles which share an half-open edge he will offer this edge and another open edge from each of those triangles. Finally, when there are no more triangles to add he can just play arbitrarily. It is not hard to verify that this is a winning strategy for Waiter.
\end{proof}
The next Lemma is rather technical. Its proof can be found in the Appendix.
\begin{lem}\label{lem: bound_core}
	There are constants $ c,L>0 $ such that w.h.p. every bi-connected component of the $ K_3 $-core of $ G_{n,p} $ is of size at most $ L $, when $ p\leq cn^{-1/2} $.
\end{lem}

We can now finish the proof of the proposition. From Lemma~\ref{lem: bound_core} we get that when $ p=cn^{-1/2} $ then w.h.p. Client wins if and only if he wins on some subgraph of $ G_{n,p} $ of order at most $ L $. On the other hand, w.h.p. every subgraph of order $ L $ will be with maximum density at most 2, which by Lemma~\ref{lem: Waiter_winning_riterion} will be Waiter's game.
\end{proof}

Lastly, we consider the Client-Waiter $ H $-game, where $ H $ is a tree. Theorem~\ref{thm:CW_random_H_biased} might lead us to think that the threshold for this game should be $ \Theta(n^{-1}) $, but in fact we will show that it is much lower. We start with a simple threshold for stars.
\begin{clm}
	Let $ S_k $ be the star with $ k $ edges. Then $ n^{-2k/(2k-1)} $ is a threshold function for Client's win in $ CW(G_{n,p},S_k,1) $.
\end{clm}
\begin{proof}
	When $ p\gg n^{-2k/(2k-1)} $, there will be w.h.p. a vertex $ v $ in $ G_{n,p} $ with degree $ 2k-1 $. Client's strategy will be to take an edge at $ v $ every time such an edge is offered to him, thus getting $ S_k $ with $ v $ at the centre.\\
	On the other hand, assume $ p\ll n^{-2k/(2k-1)}  $. Then w.h.p. every component of $ G_{n,p} $ is a tree with at most $ 2k-2 $ edges. If Waiter will play every turn on a single component then there will not be a component in Client's graph with more than $ k-1 $ edges, and in particular Client will not claim a copy of $ S_k $.
\end{proof}

On the other hand, the next claim shows that Waiter has a winning strategy in $ CW(G_{n,p},P_{k+1},1) $ provided $ p\ll n^{-2^{k/2}/(2^{k/2}-1)} $. Thus trees of the same order might have different 
thresholds.\\

\begin{clm}
	Let $ P_{k+1} $ be the path with $ k $ edges. Then Waiter wins $ CW(T_n, P_{k+1}, 1) $, where $ T_n $ is any tree of order $ n< 2^{k/2} $.
\end{clm}
\begin{proof}
Observe that $ T_n $ contains at most $ \binom{n}{2} $ copies of $ P_{k+1} $. Indeed, each path in $ T_n $ is uniquely defined by its two end points. The claim now follows from the next Waiter's winning criterion by Bednarska-Bzd\c{e}ga.
	\begin{lem}[Corollary 1.4 in \cite{CP_weight}]
		For a set $ X $ and a family of subsets $ \cl{F} $, if
		\[ \sum_{A\in\cl{F}}2^{-|A|}<\dfrac{1}{2}, \]
		then Waiter wins the $ CW(X,\cl{F},1) $ game.
	\end{lem}
\end{proof}

Though we cannot expect to find a single threshold probability for all trees of size $ k $ which depends only on $ k $, we can still show that at any rate these probabilities must be much smaller than inverse linear.
\begin{proof}[\bfseries Proof of Proposition~\ref{prp:CW_random_tree}]
	We may and will assume $ k\geq 3 $. Set $ m=(k(q+1))^2 $ and let $ T=T_{m,k} $ be the complete $ m $-ary tree of height $ k $. For an internal vertex $ x\in V(T) $ let $ A_x $ be the set of edges from $ x $ towards the leaves of $ T $, and let $ \cl{F} $ be the following family of edges,
	\[ \cl{F}=\{F\subseteq A_x: |F|=m-k+1, \text{ for some internal vertex $ x $}\}.\]
	Since $ \cl{F} $ is $ (m-k+1) $-uniform, and $ |\cl{F}|=\sum_{i=0}^{k-1}m^i\binom{m}{k-1} $ we get that
\begin{align*}
	\sum_{A\in\cl{F}}e^{-|A|/(q+1)}&\leq \binom{m}{k-1}\cdot\dfrac{m^k}{m-1}\cdot e^{-(m-k+1)/(q+1)}\\
	&\leq\dfrac{m}{m-1}\cdot\left(\dfrac{em^2}{k-1}\right)^{k-1}e^{-(m-k+1)/(q+1)}\\
	&\leq \left(\dfrac{em^2}{k-1}\right)^{k-1}e^{-(m-k)/(q+1)}.
\end{align*}
	The last expression will be smaller than 1 when 
	\[ (k-1)(q+1)(1+2\ln m-\ln (k-1))<m-k. \]
	Using our choice for $ m $ and rearranging we get that the above condition will be true when $ k\geq 3 $ and $ q\geq 1 $, thus by Theorem~\ref{thm:tool_transversal} Client has a strategy such that by the end of the game he claimed at least $ k $ out-edges of any internal vertex in $ T $, which means that he claimed a copy of $ T_{k,k} $. \\
	Since $ v_T=(m^{k+1}-1)/(m-1) $, for $ p=n^{-(m^{k+1}+1)/m^{k+1}}\gg n^{-v_T/(v_T-1)} $ w.h.p. $ G_{n,p} $ will contain $ T $ and thus $ CW(G_{n,p}, T_{k,k}, q) $ will be Client's win. 
\end{proof}

\section{Final words}\label{sec:discussion}
We have investigated several Client-Waiter games played on the edges of the complete graph. We have found that the critical bias for the maximum-degree-$ k $ game is asymptotically between $ n/k $ and $ 2n/k $ (Proposition~\ref{prp:star}). A natural question is whether either of the bounds can be improved. 
\begin{qst}
	Can we improve either bounds of Proposition~\ref{prp:star}?
\end{qst}
In the giant component game we discussed a phase transition taking place between $ n/2 $ and $ 1.6n $ (Theorem~\ref{thm:component}), where Client's achievement in this game drops from a linear-sized component to at most logarithmic. A more dramatic drop was observed in the path game: from linear length when $ q<n/2 $ to at most $ \ln\ln n $ when $ q>n $ (Theorem~\ref{thm:path}). For both games, but particularly in the giant component game, it is interesting to know if the phase transition can be more accurately located. In \cite{Man_waiter} Bednarska-Bzd\c{e}ga et al. showed that in the Waiter-Client large component game the phase transition happens around $ q=n $ (which is in accordance with the probabilistic intuition), so it is natural to expect that at least the $ 1.6n $ upper bound of the phase transition in the Client-Waiter large component game could be brought closer to $ n $, especially since we feel that Client is somewhat weaker in most games in comparison to Waiter in the corresponding Waiter-Client game.
\begin{qst}
	Can we narrow the phase transitions described in Theorem~\ref{thm:component}(ii)+(iii) and Theorem~\ref{thm:path}(ii)+(iii)?
\end{qst}

We have seen an improvement on the lower bound of the $ H $-game played on the edge set of the complete graph (i.e. Proposition~\ref{prp:k-clique lower bound}). Both the probability intuition and comparison to partial results for the Waiter-Client game (\cite{PC_H}) lead us to wonder whether the critical bias should be no higher than $ O(n^{1/m(H)}) $. An intermediate challenge could be the following.
\begin{qst}
	Is it true that for any integer $ k\geq 3 $ there is $ c>0 $ such that Waiter wins $ CW(K_n, K_k, q) $ when $ q\geq cn^{2/(k-1)} $?
\end{qst}

We have also studied the Client-Waiter $ H $-game played on the edges of the random graph. We have seen that essentially there is little difference between this game and the corresponding Maker-Breaker game, and we extended the result of Nenadov, Steger and Stojakovi\'c (\cite{MBH}) to include the biased version of the game, and the case of $ H=K_3 $ with bias at least 2. We also showed that when $ H $ is a tree the picture is more complex and highly depends on the exact structure of the tree, but nevertheless there is always some $ \epsilon>0 $ for which Client wins the game w.h.p. when $ p=n^{-1-\epsilon} $.

\textbf{Acknowledgment.} The authors would like to thank Ma\l{}gorzata Bednarska-Bzd\c{e}ga for suggesting the use of containers for the proof of Proposition~\ref{prp:k-clique lower bound}. They are also thankful to Dan Hefetz for reviewing a preliminary draft of this paper and his many helpful remarks.

\bibliographystyle{siam}
\bibliography{paper}

\begin{appendices}
	
	\section{Two missing proofs from Section~\ref{sec:H-game_on_G_n,p}}\label{apx:A}
	\begin{proof}[\bfseries The end of the proof of Proposition~\ref{prp:random_waiter_criterion}]
		Set $ k=\lfloor m_2(H)\rfloor $ and $ x=m_2(H)-k $. We consider two cases.
		\begin{enumerate}[label=\alph*)]
			\item $ 0\leq x<1/2 $. Let $ v $ be a vertex with $ d_H(v)=\delta(H) $. Since $ H $ is strictly 2-balanced we have
			\begin{align*}
			m_2(H\backslash\{v\})=\dfrac{e_H-1-\delta(H)}{v_H-3}<\dfrac{e_H-1}{v_H-2}=m_2(H),
			\end{align*}
			which leads to $ \delta(H)>m_2(H) $, and so $ \delta(H)\geq k+1 $. Suppose $ G $ is a minimal graph which contradicts the proposition. If $ v $ is a vertex with $ d_G(v)\leq 2(\delta(H)-1) $, then Waiter can play his winning strategy on $ G\backslash \{v\} $ (which exists by the minimality of $ G $) and in the last $ \delta(H)-1 $ turns offer the edges of $ v $. We have found a winning strategy for Waiter on $ G $, and that is a contradiction. Assume then that $ \delta(G)\geq 2\delta(H)-1\geq 2k+1 $. But then $ m(G)\geq k+1/2>m_2(H) $ and we have reached a contradiction again.
			\item $ x\geq 1/2 $. We consider further subcases.
			\begin{enumerate}[label=\roman*.]
				\item If $ k\geq 3 $, then, using that  $ e_H\leq\binom{v_H}{2}< \tfrac{3}{4}v_H^2-v_H $ (since $ v_H\geq 4 $), we get that
				\[ \left \lceil\dfrac{m(G)}{2}\right \rceil\leq \left \lceil\dfrac{k+1}{2}\right \rceil\leq k-1<m(H), \]
				and by Lemma~\ref{lem:dens} Waiter has a winning strategy.
				\item If $ e_H<v_H^2/4 $, then $ \tfrac{e_H}{v_H} +1/2>\tfrac{e_H-1}{v_H-2}$, and we get
				\[ \left \lceil \dfrac{m(G)}{2}\right \rceil\leq \left \lceil \dfrac{m_2(H)}{2}\right \rceil\leq  \left \lceil\dfrac{k+1}{2}\right \rceil \leq k \leq m_2(H)-1/2<m(H), \]
				and again by Lemma~\ref{lem:dens} Waiter has a winning strategy.
				\item Suppose that $ e_H\geq \lceil v_H^2/4 \rceil $ and $ k<3 $ and $ v_H\geq 5 $. Then
				\[ m_2(H)=\dfrac{e_H-1}{v_H-2}\geq \dfrac{ \lceil v_H^2/4 \rceil-1}{v_H-2}\geq 2. \]
				For any subgraph $ G'\subseteq G $, we have $e_{G'}/(v_{G'}-1)=e_{G'}/v_{G'}+e_{G'}/(v_{G'}(v_{G'}-1))\leq e_{G'}/v_{G'}+1/2 $, which together with $ m_2(H)<3 $ leads to
				\[ ar(G)\leq m(G)+1/2\leq m_2(H)+1/2<4. \]
				On the other hand, using $ m_2(H)\geq 2.5 $ and $ v_H\geq 5 $, we derive
				\[ ar(H)\geq \dfrac{e_H}{v_H-1}= \dfrac{m_2(H)(v_H-2)+1}{v_H-1} \geq 2.  \]
				We got that $ \lceil ar(G)/2\rceil<ar(H) $ and by Lemma~\ref{lem:arb} Waiter wins.
				\item The remaining case is $ v_H=4 $. In this case $ H=C_4 $ or $ H=K_4 $, as those are the only strictly 2-balanced graphs on 4 vertices. The latter can be proved by an adaptation of Lemma 2.1 in \cite{MBClique}, while in the former we have $ ar(C_4)=4/3 $, and $ ar(G)\leq m(G)+1/2\leq m_2(C_4)+1/2\leq 2 $, and again Waiter wins by Lemma~\ref{lem:arb}.
				
			\end{enumerate}
		\end{enumerate}
	\end{proof}
	
	\begin{proof}[\bfseries Proof of Lemma~\ref{lem: bound_core}]
		We call a triangle $ T $ \emph{unproblematic} if at least two of its edges are open or one of its edges is open and it shares a half-open edge with triangle $ T_1 $ which has at least one open edge. Otherwise we call $ T $ \emph{problematic}. Let $ G' $ be a bi-connected component of the $ K_3 $-core of $ G=G_{n,p} $. We describe a process to construct $ G' $ from the empty graph by repeatedly attaching triangles.\\
		
		\begin{algorithmic}[1]
			\State {Let $ T_0 $ be a triangle in $ G' $}
			\State {$ k\gets 0; \hat{G}\gets T_0 $}
			\While{$ \hat{G}\neq G' $}
				\State {$ k\gets k+1 $}
				\If{$ \hat{G} $ contains a triangle which is unproblematic in $ \hat{G} $}
				\State 	{let $ \ell<k $ be the smallest index such that $ T_\ell $ is an unproblematic triangle in $ \hat{G} $}
					\If{there is a triangle $ T\subset G' $ such that $ T $ contains one of $ T_\ell $'s open edges}
						\State {$ T_k=T $}
					\Else
						\State {let $ e\in T_\ell $ be a half-open edge of $ T_\ell $}
						\If{there is a triangle $ T\subset G' $, $ T\notin \hat{G} $ and $ T $ contains $ e $}
							\State {$ T_k=T $}
						\Else
							\State {let $ T'\subset \hat{G} $ be the other triangle which contains $ e $}
							\State {let $ T_k\subset G' $ such that $ T_k $ contains one of $ T' $'s open edges}						
						\EndIf
					\EndIf
				\Else
					\State {let $ T_k $ be a triangle in $ G' $ that is not contained in $ \hat{G} $ and intersects $ \hat{G} $ in at least one edge}
				\EndIf
				\State {$ \hat{G}\gets\hat{G}\cup T_k $}
			\EndWhile		
		\end{algorithmic}

		\noindent	We need to show that w.h.p. the highest value of $ k $ is bounded by some constant. For $ i\geq 1 $, let $ T_i $ be the triangle added to $ \hat{G} $ at the $ i $-th step, and let $ \hat{G}_i $ be the graph $ \hat{G} $ just after adding $ T_i $. If $ T_i $ intersected $ \hat{G}_{i-1} $ in exactly one edge we call $ T_i $ \emph{regular}, while if it intersected $ \hat{G}_{i-1} $ in three vertices we call it \emph{degenerate}. Denote by $ \reg(i) $ and $ \deg(i) $ the number of regular, resp. degenerate, triangles in $ \hat{G}_i $. Furthermore, for $ 1\leq i\leq \ell $ we say that $ T_i $ is \emph{fully-open} at time $ \ell $ if $ T_i $ has a vertex which is not touched by any other triangle of $ \hat{G}_\ell $ (notice that $ T_i $ is necessarily regular in this case). Denote by $ f(\ell) $ the number of fully-open triangles at time $ \ell $. 
		\begin{clm}\label{clm:fully_open}
			For every $ \ell\geq 1 $, assuming the process does not stop before the $ \ell $-th step, we have
			\[ f(\ell)\geq \frac{1}{2}\cdot\reg(\ell)-3\cdot\deg(\ell). \]
		\end{clm}
		\begin{proof}
			Denote the right hand side of the above by $ \varphi(\ell):=\reg(\ell)/2-3\deg(\ell) $. We will use induction to show that the following stronger statement holds for any $ \ell\geq 1 $
			\[ f(\ell)\geq\begin{cases}
			\varphi(\ell), & \text{if } T_\ell \text{ is regular}\\
			\varphi(\ell)+1, & \text{if } T_\ell \text{ is degenerate.}
			\end{cases} \]
			This is true for $ \ell=1 $ since $ T_1 $ must be regular and $ f(1)=1>1/2 $. At $ \ell=2 $ triangle $ T_0 $ still has two open edges, hence $ T_2 $ is regular as well and $ f(2)=2> 1 $. Suppose now that we are at the $ \ell $-th step, $ \ell\geq 3 $. If $ T_\ell $ is degenerate, then since $ T_\ell $ shares an edge with at most two fully-open triangles we have $ f(\ell)\geq f(\ell-1)-2\geq\reg(\ell-1)/2-3\deg(\ell-1)-2=\reg(\ell)/2-3(\deg(\ell)-1)-2=\varphi(\ell)+1 $. Otherwise, assume that $ T_\ell $ is regular. Consider two cases
			\begin{itemize}
				\item If $ T_\ell $ does not connect to a fully-open triangle then since $ T_\ell $ is regular and fully-open, $ f(\ell)=f(\ell-1)+1\geq \varphi(\ell-1)+1=\varphi(\ell)+1/2 $.
				\item If $ T_\ell $ does connect to a fully-open triangle then $ f(\ell)=f(\ell-1) $. If $ T_{\ell-1} $ was degenerate then $ f(\ell-1)\geq \varphi(\ell-1)+1=\varphi(\ell)+1/2 $. Assume then that $ T_{\ell-1} $ is regular and connected to $ T' $. If $ T' $ was not fully open then $ f(\ell)=f(\ell-2)+1\geq\varphi(\ell-2)+1\geq (\reg(\ell)-2)/2-3\deg(\ell)+1=\varphi(\ell) $. Assume then that $ T' $ was fully open. If $ T_{\ell-2} $ was degenerate then $ f(\ell)=f(\ell-2)\geq\varphi(\ell-2)+1=\varphi(\ell) $. Assume that $ T_{\ell-2} $ was regular and connected to $ T'' $. Again if $ T'' $ was not fully open we are done, otherwise $ T'' $ is fully open and $ T_{\ell-2} $ is connected to one of its open edges, $ e $. At time $ \ell-1 $, $ T'' $ is still problematic since it has an open edge and $ e $ is half-open and $ T_{\ell-2} $ is fully open. Then by our algorithm, and since $ T' $ is fully open, it must be that $ T'=T_{\ell-2} $. But then at time $ \ell $, $ T'' $ is still problematic and therefore $ T_\ell $ must connect to $ T_{\ell-2} $. But $ T_{\ell-2} $ is no longer fully open, which is a contradiction.
			\end{itemize}
		\end{proof}
		Returning to the proof of the lemma, suppose we are in the $ i $-th step of the process described above. We first bound the probabilities of finding certain triangles. If there is an unproblematic triangle and we are about to add a regular triangle, then there are at most four edges to which we may connect, and at most $ n $ possibilities for the extra vertex. Thus we bound, $ \Pr_{reg, unprob.}\leq 4np^2\leq 4c^2<1/2 $, if we choose $ c<1/\sqrt{8} $. A degenerate triangle just adds at least one new edge to the graph, which at the $ i $-th step has at most $ 3i $ vertices. Thus $ P_{deg}(i)\leq\binom{3i}{3}p\leq (3i)^3cn^{-1/2} $. Set $ L=42 $ and $ \ell_0=4\log_2n $. Let $ X $ be the random variable of the number of sequences, when running this process on $ G_{n,p} $, that contain at least 7 degenerate triangles in the first $ \ell_0 $ steps. After the $ L $-th step, and as long as we have less then 7 degenerate triangles, we must have unproblematic triangles. Because for any $ \ell>L $, by Claim~\ref{clm:fully_open}  $ f(\ell)\geq (\ell-6)/2-18>0 $. So the probability for a regular triangle after the $ L $-th step and before the 7-th degenerate triangle is at most $ 1/2 $. Denote by $ \ell' $ the moment in which the 7-th degenerate triangle appears, then we have
		\begin{align*}
		\EE[X]&\leq \binom{n}{3}\sum_{8\leq \ell'\leq \ell_0}\binom{\ell'-1}{6}(27\ell_0^3cn^{-1/2})^7\cdot 2^{-(\ell'-L-6)}=o(1).
		\end{align*} 
		Now denote by $ Y $ the random variable of the number of sequences that last more than $ \ell_0 $ steps and contain less than 7 degenerate triangles in the first $ \ell_0 $ steps. We get
		\begin{align*}
		\EE[Y]\leq \binom{n}{3}\sum_{k=0}^{6}\binom{\ell_0}{k}(27\ell_0^3cn^{-1/2})^k\cdot 2^{-(\ell_0-L-k)}=o(1),
		\end{align*}
		by our choice of $ \ell_0 $. So w.h.p. $ X=Y=0 $, which means that all the processes last for less than $ \ell_0 $ steps and contain at most 6 degenerate triangles. Denote by $ \ell_e $ the length of such process. Then it must be that $ f(\ell_e)=0 $, or we would still have fully-open triangles at step $ \ell_e $. Thus by Claim~\ref{clm:fully_open}
		\begin{align*}
		0=f(\ell_e)\geq \reg(\ell_e)/2-3\deg(\ell_e)\geq (\ell_e-6)/2-18\Longrightarrow \ell_e\leq 42.
		\end{align*}
	\end{proof}
	
\end{appendices}
\end{document}